\documentclass[12pt]{article}
\usepackage{amsmath,amsthm,amsfonts,amssymb}
\usepackage{fullpage}
\usepackage{multirow}
\usepackage{array}
\usepackage{bbm}
\usepackage[noblocks]{authblk}
\usepackage{verbatim}
\usepackage{tikz}
\usetikzlibrary{math}
\usepackage{float}
\usepackage{enumerate}
\usepackage{diagbox}
\usepackage{soul}
\usepackage{cite}

\newtheorem{theorem}{Theorem}[section]
\newtheorem{lemma}[theorem]{Lemma}

\newtheorem{corollary}[theorem]{Corollary}

\theoremstyle{definition}
\newtheorem{definition}[theorem]{Definition}

\newtheorem{question}[theorem]{Question}
\newtheorem{remark}[theorem]{Remark}

\newlength{\Oldarrayrulewidth}

\renewcommand{\i}{\textup{\textbf{i}}}
\renewcommand{\j}{\textup{\textbf{j}}}

\newcommand{\p}{\textup{\textbf{p}}}
\newcommand{\q}{\textup{\textbf{q}}}

\newcommand{\lcm}{\textup{lcm}}

\renewcommand{\mod}[2]{#1\textup{ (mod }#2\textup{)}}
\renewcommand{\bmod}[2]{\bf#1\textup{\bf (mod}\,#2\textup{)}}

\def\m@th{\mathsurround=0pt}
\def\sm#1{\null\,\vcenter{\baselineskip9pt\lineskip.23ex\m@th
    \ialign{\hfil$\scriptstyle##$\hfil&&\ \hfil$\scriptstyle##$\hfil\crcr
    \mathstrut\crcr\noalign{\kern-\baselineskip}
    #1\crcr\mathstrut\crcr\noalign{\kern-\baselineskip}}}\,}
\def\smnp#1{\null\,\vcenter{\baselineskip9pt\lineskip.23ex\m@th
    \ialign{\hfil$\scriptstyle##$\hfil&&\ \ \hfil$\scriptstyle##$\hfil\crcr
    \mathstrut\crcr\noalign{\kern-\baselineskip}
    #1\crcr\mathstrut\crcr\noalign{\kern-\baselineskip}}}\,}

\begin{document}

\title{Covering systems with odd moduli}
\author[1]{Joshua~Harrington\thanks{joshua.harrington@cedarcrest.edu}}
\author[2]{Yewen~Sun\thanks{yewen@ucsb.edu}}
\author[3]{Tony~W.~H.~Wong\thanks{wong@kutztown.edu}}
\affil[1]{Department of Mathematics, Cedar Crest College}
\affil[2]{Department of Mathematics, University of California, Santa Barbara}
\affil[3]{Department of Mathematics, Kutztown University of Pennsylvania}
\date{\today}

\maketitle

\begin{abstract}
The concept of a covering system was first introduced by Erd\H{o}s in 1950.  Since their introduction, a lot of the research regarding covering systems has focused on the existence of covering systems with certain restrictions on the moduli.  Arguably, the most famous open question regarding covering systems is the odd covering problem.  In this paper, we explore a variation of the odd covering problem, allowing a single odd prime to appear as a modulus in the covering more than once, while all other moduli are distinct, odd, and greater than $1$.  We also consider this variation while further requiring the moduli of the covering system to be square-free.
\end{abstract}

\section{Introduction}\label{sec:intro}
In 1950, Erd\H{o}s \cite{erdos} introduced the concept of a covering system of the integers, which is defined as follows.

\begin{definition}
A \emph{covering system of the integers} is a finite collection of congruences such that every integer satisfies at least one of the congruences in the collection.
\end{definition}

Since their introduction, the existence of certain types of covering systems and their applications has been the investigation of many articles \cite{ffk2008,ffk,ffkpy,filharr,fh,fj,hough,jw,n,o}.  Of particular interest to the investigation of this paper is the speculative existence of an odd covering.  The following is a well known question, originally asked by Erd\H{o}s.

\begin{question}[The Odd Covering Problem]
Does there exist a covering system of the integers whose moduli are all odd, distinct, and greater than 1?
\end{question}

If the answer to this question is yes, then such a covering system would be called an \emph{odd covering system of the integers}.  Hough and Nielsen \cite{hn} showed in 2019 that if the moduli of a covering system are distinct and greater than 1, then at least one of the moduli of the system must be divisible by either $2$ or $3$.  

If an odd covering exists such that all moduli of the system are square-free, then we refer to this as a \emph{square-free odd covering system of the integers}.  In 2005, Guo and Sun \cite{gs} showed that the lowest common multiple of all moduli in a square-free odd covering system (if it exists) must have at least 22 distinct prime factors.%In 2005, Song and Sun \cite{} showed that if an odd covering system of the integers exists where all moduli of the system are square-free, then the LCM of the moduli of the system must have at least 22 distinct prime factors.  If an odd covering exists such that all of the moduli of the system are square-free, then we refer to this as a \emph{square-free odd covering system of the integers}.

In 2015, Harrington \cite{h} posed the following question related to the odd covering problem.

\begin{question}
Given an odd integer $n\geq 3$, does there exist a covering system of the integers that has $n$ as a modulus at most twice, such that all other moduli are odd, distinct, and greater than 1?
\end{question}

Harrington answered this question in the affirmative when $n=3$ in the same paper.  The following more general question was later asked by Hammer, Harrington, and Marotta \cite{hhm} in 2018.

\begin{question}\label{question:mainquestion}
Given a fixed odd prime $p$, what is the smallest nonnegative integer $t_p$ for which there exists a covering system of the integers that has $p$ as a modulus exactly $t_p$ times, such that all other moduli are odd, distinct, and greater than 1?
\end{question}

If an odd covering system of the integers exists, then the answer to Question~\ref{question:mainquestion} is that $t_p\leq 1$ for all odd primes $p$.  Harrington's result yields $t_3\leq 2$ and Hammer, Harrington, and Marotta showed that $t_5\leq 3$.  One may further ask a similar question regarding square-free odd coverings.

%In Section~\ref{sec:squarefree} of this paper, we make a connection between Question~\ref{question:mainquestion}
%and the following analogue for square-free odd coverings.

\begin{question}\label{question:mainquestionsquarefree}
Given a fixed odd prime $p$, what is the smallest nonnegative integer $\tau_p$ for which there exists a covering system of the integers that has $p$ as a modulus exactly $\tau_p$ times, such that all other moduli are odd, distinct, square-free, and greater than 1?
\end{question}

In Theorem~\ref{thm:3sqrfreeimpliesodd} of Section~\ref{sec:squarefree}, we establish that if $\tau_p\leq 2$ for any odd prime $p$, then there exists an odd covering of the integers.  We also show in Section~\ref{sec:squarefree} that $\tau_7\leq 6$.  In Section~\ref{sec:main}, we show that $t_7\leq 4$, $t_{11}\leq 7$, and $t_p\leq p-5$ for all primes $p\geq23$.  We conclude our investigation in Section~\ref{sec:twolemmas} by introducing a lemma that allows us to generalize the main results in Sections~\ref{sec:squarefree} and \ref{sec:main}.

\section{Condensed tree diagram and notation}\label{sec:tree}

We will adopt the notation used in \cite{filharr}.  Given lists $[r_1,r_2,\ldots,r_k]$ and $[m_1,m_2,\ldots,m_k]$ with $m_1,m_2\ldots,m_k$ pairwise relatively prime positive integers, we let 
$$([r_1,r_2,\ldots,r_k],[m_1,m_2,\ldots,m_k])$$
denote the congruence system $x\equiv\mod{r_j}{m_j}$, $1\leq j\leq k$. By the Chinese Remainder Theorem, this congruence system is equivalent to the congruence $x\equiv\mod{r}{m}$, where $m=m_1m_2\cdots m_k$ and $r\in\{0,1,\dotsc,m-1\}$. Hence, we sometimes also refer to this congruence system as simply a congruence. With this notation, we provide the following example of a covering system.
\begin{equation}\label{eq:example}\begin{gathered}
([1],[2]),\ ([0,1],[2,3]),\ ([0,2],[2,3]),\ ([0,3],[2,9]),\\
([0,0,0],[2,9,5]),\ ([0,0,1],[2,9,5]),\ ([0,0,2],[2,9,5]),\ ([0,0,3],[2,9,5]),\ ([0,0,4],[2,9,5]),\\
([0,6,0],[2,9,5]),\ ([0,6,1],[2,9,5]),\ ([0,6,2],[2,9,5]),\ ([0,6,3],[2,9,5]),\ ([0,6,4],[2,9,5]).
\end{gathered}\end{equation}

This is a covering system since we first cover all integers that are congruent to $1$ modulo $2$ with $([1],[2])$, then all integers congruent to $0$ modulo $2$ are split into congruence classes $0,1,2$ modulo $3$, covered by $([0,0],[2,3])$, $([0,1],[2,3])$, $([0,2],[2,3])$. The congruences $([0,1],[2,3])$ and $([0,2],[2,3])$ are explicit in \eqref{eq:example}, while integers covered by $([0,0],[2,3])$ are further split into congruence classes $0,3,6$ modulo $9$, covered by $([0,0],[2,9])$, $([0,3],[2,9])$, $([0,6],[2,9])$. Once again, the congruence $([0,3],[2,9])$ is explicit in \eqref{eq:example}, while integers covered by each of $([0,0],[2,9])$ and $([0,6],[2,9])$ are split into $0,1,2,3,4$ modulo $5$.

The covering system given in \eqref{eq:example}, however, has many repeated moduli. To eliminate repeated moduli, we forgo some of the congruences inside a congruence system. For example, we forgo $\mod{0}{2}$ in $([0,1],[2,3])$ and write it as $([1],[3])$, and we forgo $\mod{0}{2}$ and $\mod{0}{9}$ in $([0,0,0],[2,9,5])$ to get $([0],[5])$. When we forgo some congruences inside a congruence system, the set of integers covered by the new congruence system is a superset of the original one. Hence, the covering system given by \eqref{eq:example} can be modified to obtain the following covering system.
\begin{equation}\label{eq:example2}\begin{gathered}
([1],[2]),\quad ([1],[3]),\quad ([0,2],[2,3]),\quad ([3],[9]),\\
([0],[5]),\quad ([0,1],[2,5]),\quad ([0,2],[3,5]),\quad ([0,0,3],[2,3,5]),\quad ([0,4],[9,5]),\\
([0],[5]),\quad ([0,1],[2,5]),\quad ([0,2],[3,5]),\quad ([0,0,3],[2,3,5]),\quad ([0,6,4],[2,9,5]).
\end{gathered}\end{equation}
Notice that several congruence systems are repeated in \eqref{eq:example2}, thus we can simplify the covering system as follows.
\begin{equation*}\label{eq:example3}\begin{gathered}
([1],[2]),\quad ([1],[3]),\quad ([0,2],[2,3]),\quad ([3],[9]),\quad ([0],[5]),\\
([0,1],[2,5]),\quad ([0,2],[3,5]),\quad ([0,0,3],[2,3,5]),\quad ([0,4],[9,5]),\quad ([0,6,4],[2,9,5]).
\end{gathered}\end{equation*}
This corresponds to the following set of congruences.%The set of congruences in \eqref{eq:example} forms a covering system since every integer satisfies one of 
\begin{equation}\label{eq:example4}
\begin{gathered}
\mod{1}{2},\quad \mod{1}{3},\quad \mod{2}{6},\quad \mod{3}{9},\quad \mod{0}{5},\\
\mod{6}{10},\quad \mod{12}{15},\quad \mod{18}{30},\quad \mod{9}{45},\quad \mod{24}{90}.
\end{gathered}\end{equation}

To visualize the covering system given by \eqref{eq:example}, we represent the process of splitting integers into congruence classes using a tree diagram, as illustrated in Figure~\ref{fig:example}. We refer to the node at the top as the \emph{root}, and the nodes with no further branches below as the \emph{leaves}. The congruences given at the leaves of this tree come from \eqref{eq:example4}.
\begin{figure}[H]
\centering
\begin{tikzpicture}[scale=1.05]
\draw(-2,-1)--(0,0)--(2,-1);\node[below]at(2,-1){$\footnotesize\substack{\bmod{1}{2}}$};
\draw(-1.075,-0.35)node[rotate={atan(1/2)}]{$\footnotesize\substack{([0],[2])}$};\draw (1.075,-0.35) node [rotate={atan(-1/2)}]{$\footnotesize\substack{([1],[2])}$};
\draw(-4,-2.5)--(-2,-1)--(0,-2.5);\draw(-2,-1)--(-2,-2.5);\node[below]at(-2,-2.4){$\footnotesize\substack{\bmod{1}{3}}$};\node[below]at(0,-2.4){$\footnotesize\substack{\bmod{2}{6}}$};
\draw(-3.25,-1.75)node[rotate={atan(3/4)}]{$\footnotesize\substack{([0,0],[2,3])}$};\draw(-2.15,-1.85)node[rotate=90]{$\footnotesize\substack{([0,1],[2,3])}$};\draw (-0.75,-1.75) node [rotate={atan(-3/4)}]{$\footnotesize\substack{([0,2],[2,3])}$};
%\draw(-6,-3)--(-4,-2)--(-2,-3);\draw(-4,-2)--(-4,-3);\node[below]at(-4,-3){$\footnotesize\substack{\mod{3}{9}}$};
%\draw(-5.2,-2.4)node[rotate={atan(1/2)}]{$\footnotesize\substack{\mod{0}{9}}$};\draw(-4.2,-2.5)node[rotate=90]{$\footnotesize\substack{\mod{3}{9}}$};\draw (-2.8,-2.4) node [rotate={atan(-1/2)}]{$\footnotesize\substack{\mod{6}{9}}$};
\draw(-8,-4)--(-4,-2.5)--(0,-4);\draw(-4,-2.5)--(-4,-4);\node[below]at(-4,-4){$\footnotesize\substack{\bmod{3}{9}}$};
\draw(-6.45,-3.25)node[rotate={atan(1.5/4)}]{$\footnotesize\substack{([0,0],[2,9])}$};\draw(-4.15,-3.3)node[rotate=90]{$\footnotesize\substack{([0,3],[2,9])}$};\draw(-1.55,-3.25)node[rotate={atan(-1.5/4)}]{$\footnotesize\substack{([0,6],[2,9])}$};
\draw(-11,-6)--(-8,-4)--(-5,-6);\draw(-9.5,-6)--(-8,-4)--(-6.5,-6);\draw(-8,-4)--(-8,-6);\node[below]at(-11,-6){$\footnotesize\substack{\bmod{0}{5}}$};\node[below]at(-9.5,-6){$\footnotesize\substack{\bmod{6}{10}}$};\node[below]at(-8,-6){$\footnotesize\substack{\bmod{12}{15}}$};\node[below]at(-6.5,-6){$\footnotesize\substack{\bmod{18}{30}}$};\node[below]at(-5,-6){$\footnotesize\substack{\bmod{9}{45}}$};
\draw(-9.8,-5)node[rotate={atan(2/3)}]{$\footnotesize\substack{([0,0,0],[2,9,5])}$};\draw(-9.1,-5.2)node[rotate={atan(4/3)}]{$\footnotesize\substack{([0,0,1],[2,9,5])}$};\draw(-8.15,-5.2)node[rotate=90]{$\footnotesize\substack{([0,0,2],[2,9,5])}$};\draw(-6.9,-5.2)node[rotate={atan(-4/3)}]{$\footnotesize\substack{([0,0,3],[2,9,5])}$};\draw(-6.2,-5)node[rotate={-atan(2/3)}]{$\footnotesize\substack{([0,0,4],[2,9,5])}$};
\draw(-3,-6)--(0,-4)--(3,-6);\draw(-1.5,-6)--(0,-4)--(1.5,-6);\draw(0,-4)--(0,-6);\node[below]at(-3,-6){$\footnotesize\substack{\bmod{0}{5}}$};\node[below]at(-1.5,-6){$\footnotesize\substack{\bmod{6}{10}}$};\node[below]at(0,-6){$\footnotesize\substack{\bmod{12}{15}}$};\node[below]at(1.5,-6){$\footnotesize\substack{\bmod{18}{30}}$};\node[below]at(3,-6){$\footnotesize\substack{\bmod{24}{90}}$};
\draw(-1.8,-5)node[rotate={atan(2/3)}]{$\footnotesize\substack{([0,6,0],[2,9,5])}$};\draw(-1.1,-5.2)node[rotate={atan(4/3)}]{$\footnotesize\substack{([0,6,1],[2,9,5])}$};\draw(-0.15,-5.2)node[rotate=90]{$\footnotesize\substack{([0,6,2],[2,9,5])}$};\draw(1.1,-5.2)node[rotate={atan(-4/3)}]{$\footnotesize\substack{([0,6,3],[2,9,5])}$};\draw(1.8,-5)node[rotate={-atan(2/3)}]{$\footnotesize\substack{([0,6,4],[2,9,5])}$};
\end{tikzpicture}
\caption{A tree diagram of a covering system showing all moduli and residues}
\label{fig:example}
\end{figure}
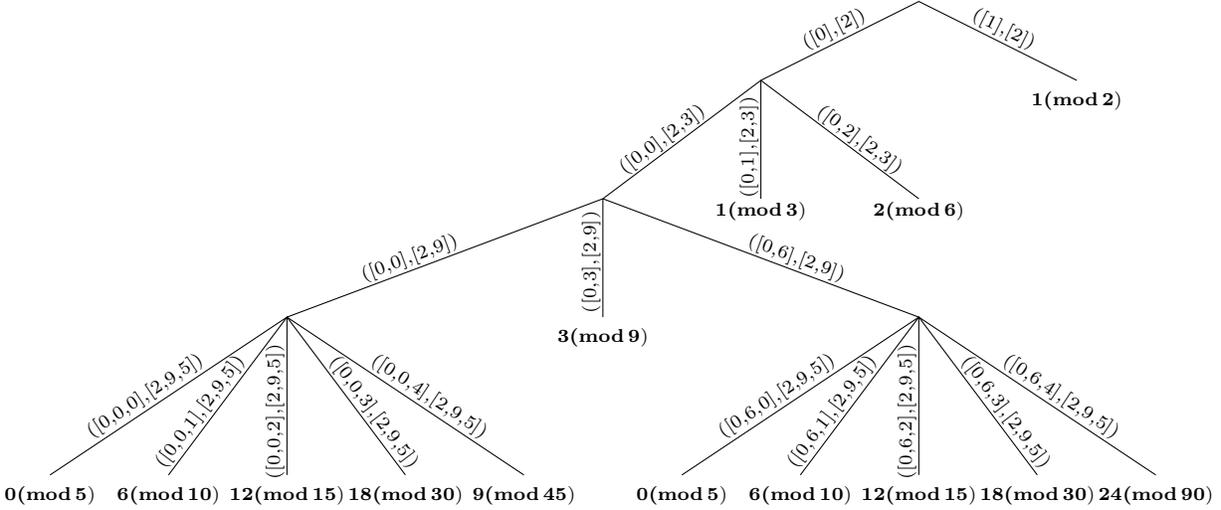
%%%%%%%%%%%%%%%%%%%%%%%%%%%

%In the tree diagram presented in Figure~\ref{fig:example}, every node of the tree represents a certain subset $S$ of the integers. At each non-leaf node, there are always a prime number $p$ of branches immediately below that node, and the branches represent a partition of $S$ into congruence classes modulo $p^\alpha$ for some positive integer $\alpha$.
The branches immediately below a node in the tree diagram are called the \emph{child branches} of that node. At each non-leaf node, the number of child branches is always given by a prime number $p$, and we refer to such a node as a \emph{$p$-node}. Each $p$-node represents a certain subset $S$ of the integers, and its child branches represent a partition of $S$ into congruence classes modulo $p^\alpha m'$ for some positive integers $\alpha$ and $m'$, where $\alpha$ is equal to the number of $p$-nodes along the path from the root and $p\nmid m'$. If the $p$-node represents the subset of integers satisfying the congruence $([r,r'],[p^{\alpha-1},m'])$, then we arrange the congruences such that $([r,r'],[p^\alpha,m'])$, $([r+p^{\alpha-1},r'],[p^\alpha,m'])$, $([r+2p^{\alpha-1},r'],[p^\alpha,m'])$, $\dotsc$, $([r+(p-1)p^{\alpha-1},r'],[p^\alpha,m'])$ are placed from left to right among the child branches. In other words, the congruences on the branches are implied by the positioning of the branches and thus can be omitted from the diagram, as illustrated in Figure~\ref{fig:example1}.
\begin{figure}[H]
\centering
\begin{tikzpicture}[scale=1.05]
\draw(-2,-1)--(0,0)--(2,-1);\node[below]at(2,-1){$\footnotesize\substack{\bmod{1}{2}}$};
\draw(-4,-2)--(-2,-1)--(0,-2);\draw(-2,-1)--(-2,-2);\node[below]at(-2,-2){$\footnotesize\substack{\bmod{1}{3}}$};\node[below]at(0,-2){$\footnotesize\substack{\bmod{2}{6}}$};
\draw(-8,-3)--(-4,-2)--(0,-3);\draw(-4,-2)--(-4,-3);\node[below]at(-4,-3){$\footnotesize\substack{\bmod{3}{9}}$};
\draw(-11,-4)--(-8,-3)--(-5,-4);\draw(-9.5,-4)--(-8,-3)--(-6.5,-4);\draw(-8,-3)--(-8,-4);\node[below]at(-11,-4){$\footnotesize\substack{\bmod{0}{5}}$};\node[below]at(-9.5,-4){$\footnotesize\substack{\bmod{6}{10}}$};\node[below]at(-8,-4){$\footnotesize\substack{\bmod{12}{15}}$};\node[below]at(-6.5,-4){$\footnotesize\substack{\bmod{18}{30}}$};\node[below]at(-5,-4){$\footnotesize\substack{\bmod{9}{45}}$};
\draw(-3,-4)--(0,-3)--(3,-4);\draw(-1.5,-4)--(0,-3)--(1.5,-4);\draw(0,-3)--(0,-4);\node[below]at(-3,-4){$\footnotesize\substack{\bmod{0}{5}}$};\node[below]at(-1.5,-4){$\footnotesize\substack{\bmod{6}{10}}$};\node[below]at(0,-4){$\footnotesize\substack{\bmod{12}{15}}$};\node[below]at(1.5,-4){$\footnotesize\substack{\bmod{18}{30}}$};\node[below]at(3,-4){$\footnotesize\substack{\bmod{24}{90}}$};
\end{tikzpicture}
\caption{A tree diagram of a covering system with congruences on the branches omitted}
\label{fig:example1}
\end{figure}
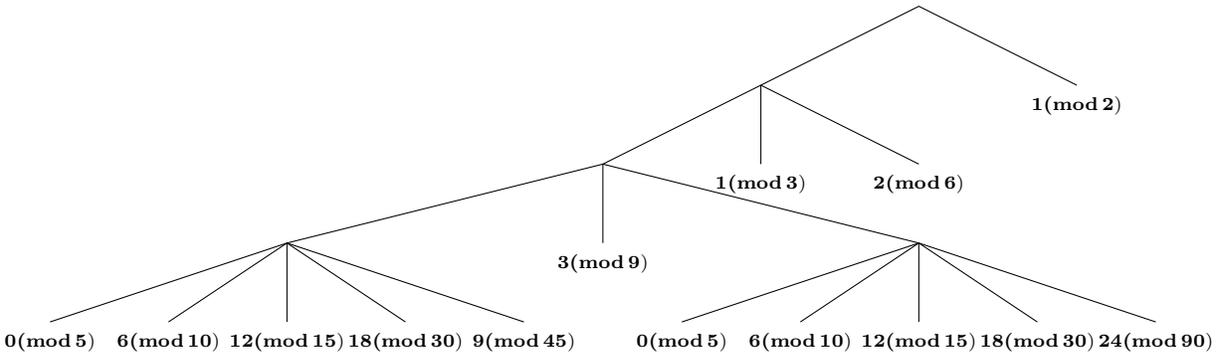

Next, consider a fixed leaf in the tree diagram. Recall that the congruence is originally obtained by combining all congruences along the path from the root using the Chinese Remainder Theorem. We then choose a subset of these congruences along the path to form the congruence displayed at the leaf. Note that the residue in this congruence is implied by the modulus of the leaf since the factors in the modulus determine which congruences along the path were used. Hence, the residue at the leaf can also be omitted from the diagram, as illustrated in Figure~\ref{fig:example2}.
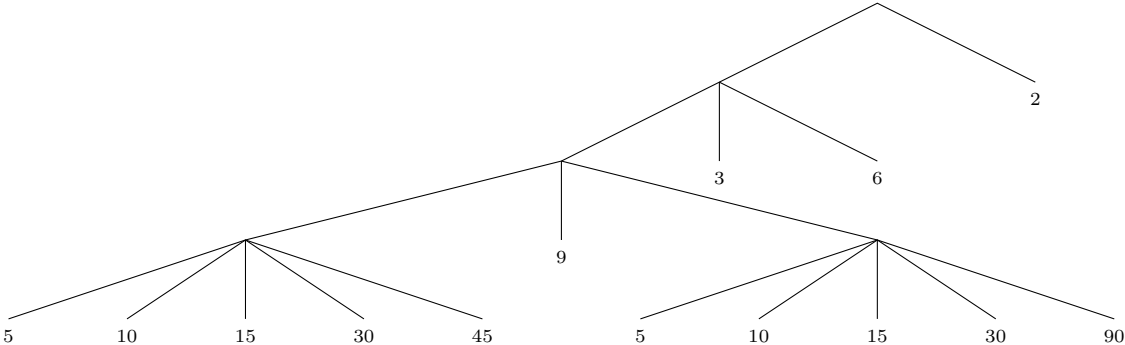
\begin{figure}[H]
\centering
\begin{tikzpicture}[scale=1.05]
\draw(-2,-1)--(0,0)--(2,-1);\node[below]at(2,-1){$\footnotesize\substack{2}$};
\draw(-4,-2)--(-2,-1)--(0,-2);\draw(-2,-1)--(-2,-2);\node[below]at(-2,-2){$\footnotesize\substack{3}$};\node[below]at(0,-2){$\footnotesize\substack{6}$};
\draw(-8,-3)--(-4,-2)--(0,-3);\draw(-4,-2)--(-4,-3);\node[below]at(-4,-3){$\footnotesize\substack{9}$};
\draw(-11,-4)--(-8,-3)--(-5,-4);\draw(-9.5,-4)--(-8,-3)--(-6.5,-4);\draw(-8,-3)--(-8,-4);\node[below]at(-11,-4){$\footnotesize\substack{5}$};\node[below]at(-9.5,-4){$\footnotesize\substack{10}$};\node[below]at(-8,-4){$\footnotesize\substack{15}$};\node[below]at(-6.5,-4){$\footnotesize\substack{30}$};\node[below]at(-5,-4){$\footnotesize\substack{45}$};
\draw(-3,-4)--(0,-3)--(3,-4);\draw(-1.5,-4)--(0,-3)--(1.5,-4);\draw(0,-3)--(0,-4);\node[below]at(-3,-4){$\footnotesize\substack{5}$};\node[below]at(-1.5,-4){$\footnotesize\substack{10}$};\node[below]at(0,-4){$\footnotesize\substack{15}$};\node[below]at(1.5,-4){$\footnotesize\substack{30}$};\node[below]at(3,-4){$\footnotesize\substack{90}$};
\end{tikzpicture}
\caption{A tree diagram of a covering system showing only the moduli}
\label{fig:example2}
\end{figure}

Continuing to simplify the tree diagram, if the moduli at several leaves branching from a node share the same factor $m_0$, then we can merge them to form a wedge, and we write $\{m_1,m_2,\dotsc,m_k\}\times m_0$ to indicate $2^k$ moduli, given by the product of any subset of $\{m_1,m_2,\dotsc,m_k\}$ together with $m_0$. These congruences are arranged so that they increase in magnitude from left to right among the child branches. This allows us to obtain a condensed tree diagram in Figure~\ref{fig:example3}.

\begin{figure}[H]
\centering
\begin{tikzpicture}[scale=1.2]
\draw(-2,-1)--(0,0)--(2,-1);\node[below]at(2,-1){$\substack{2}$};
\draw(-2,-1)--(0,-1.25)--(-0.25,-1.5)--cycle;\node[right]at(-0.125,-1.375){$\substack{\{2\}\times3\\2\text{ branches}}$};\draw(-2,-1)--(-2,-2);
\draw(-4,-3)--(-2,-2)--(2,-3);\draw(-2,-2)--(-2,-3);\node[below]at(-2,-3){$\substack{3^2}$};
\draw(-4,-3)--(-6,-3.25)--(-5.75,-3.5)--cycle;\node[left]at(-6,-3.375){$\substack{\{2,3\}\times5\\4\text{ branches}}$};\draw(-4,-3)--(-4,-4);\node[below]at(-4,-4){$\substack{3^2\times5}$};
\draw(2,-3)--(0,-3.25)--(0.25,-3.5)--cycle;\node[left]at(0,-3.375){$\substack{\{2,3\}\times5\\4\text{ branches}}$};\draw(2,-3)--(2,-4);\node[below]at(2,-4){$\substack{2\times3^2\times5}$};
\end{tikzpicture}
\caption{A condensed tree diagram of the covering system given by \eqref{eq:example4}}
\label{fig:example3}
\end{figure}
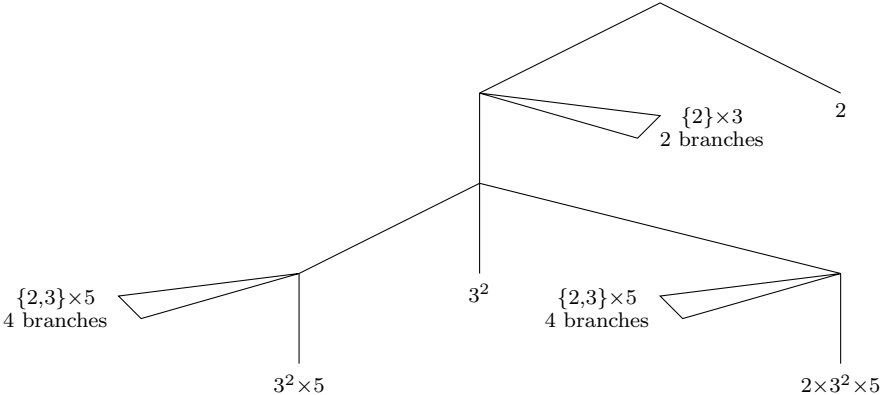

To illustrate a common technique used when building covering systems of the integers, we now turn our attention to a different system. Let $p>3$ be a prime and consider the following tree diagram.

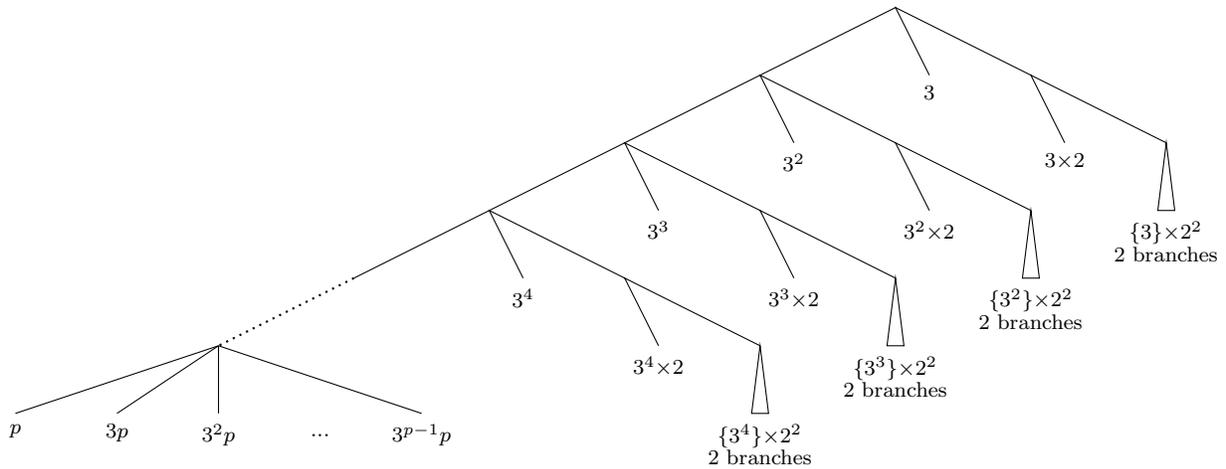
\begin{figure}[H]
\centering
\begin{tikzpicture}[scale=0.9]
\foreach\i in{0}{\draw(2-2*\i,-1-\i)--(-2*\i,-\i)--(0.5-2*\i,-1-\i);\tikzmath{\x=int(3^(\i+1));}\node[below]at(0.5-2*\i,-1-\i){$\substack{\x}$};
\draw(4-2*\i,-2-\i)--(2-2*\i,-1-\i)--(2.5-2*\i,-2-\i);\node[below]at(2.5-2*\i,-2-\i){$\substack{\x\times2}$};
\draw(4-2*\i,-2-\i)--(3.875-2*\i,-3-\i)--(4.125-2*\i,-3-\i)--cycle;\node[below]at(4-2*\i,-3-\i){$\substack{\{3\}\times2^2\\2\text{ branches}}$};}
%\draw(4-2*\i,-2-\i)--(6-2*\i,-2.25-\i)--(5.75-2*\i,-2.5-\i)--cycle;\tikzmath{\x=int(3^(\i+1));}\node[right]at(6-2*\i,-2.375-\i){$\substack{\{\x\}\times2^2\\2\text{ branches}}$};}
\foreach[count=\i]\j in{2,3,4}{\draw(2-2*\i,-1-\i)--(-2*\i,-\i)--(0.5-2*\i,-1-\i);\node[below]at(0.5-2*\i,-1-\i){$\substack{3^\j}$};
\draw(4-2*\i,-2-\i)--(2-2*\i,-1-\i)--(2.5-2*\i,-2-\i);\node[below]at(2.5-2*\i,-2-\i){$\substack{3^\j\times2}$};
\draw(4-2*\i,-2-\i)--(3.875-2*\i,-3-\i)--(4.125-2*\i,-3-\i)--cycle;\node[below]at(4-2*\i,-3-\i){$\substack{\{3^\j\}\times2^2\\2\text{ branches}}$};}
%\draw(4-2*\i,-2-\i)--(6-2*\i,-2.25-\i)--(5.75-2*\i,-2.5-\i)--cycle;\tikzmath{\x=int(3^(\i+1));}\node[right]at(6-2*\i,-2.375-\i){$\substack{\{\x\}\times2^2\\2\text{ branches}}$};}
\draw(0,0)--(-8,-4);\draw[dotted,thick](-8,-4)--(-10,-5);
\foreach\i in{0,1,2,4}{\draw(-10,-5)--(-13+1.5*\i,-6);}\node[below]at(-13,-6){$\substack{p}$};\node[below]at(-11.5,-6){$\substack{3p}$};\node[below]at(-10,-6){$\substack{3^2p}$};\node[below]at(-8.5,-6.1){$\substack{\dotsb}$};\node[below]at(-7,-6){$\substack{3^{p-1}p}$};
\end{tikzpicture}
\caption{A condensed tree diagram of a covering system whose smallest modulus is $3$}
\label{fig:highpower}
\end{figure}

In Figure~\ref{fig:highpower}, note that the structure of the subtree given in
Figure~\ref{fig:subtree} is repeated for $i\in\{1,2,\dotsc,p-1\}$. As we consistently split the leftmost node to attach a repeated subtree, we call the leftmost branch from the root a ``power branch" of the tree, and condense it as in Figure~\ref{fig:highpowercondensed}. To convert Figure~\ref{fig:highpowercondensed} back to Figure~\ref{fig:highpower}, the $i$-th repeated subtree is obtained by substituting every factor of $3$ in the first repeated subtree with $3^i$ for $2\leq i\leq p-1$.

\begin{figure}[H]
\centering
\begin{minipage}{0.45\textwidth}
\centering
\begin{tikzpicture}[scale=0.85]
\draw(2,-1)--(0,0)--(0.5,-1);\node[below]at(0.5,-1){$\substack{3^i}$};
\draw(4,-2)--(2,-1)--(2.5,-2);\node[below]at(2.5,-2){$\substack{3^i\times2}$};
\draw(4,-2)--(3.875,-3)--(4.125,-3)--cycle;\node[below]at(4,-3){$\substack{\{3^i\}\times2^2\\2\text{ branches}}$};
\end{tikzpicture}
\caption{Repeated subtree from Figure~\ref{fig:highpower}}
\label{fig:subtree}
\end{minipage}\qquad
\begin{minipage}{0.45\textwidth}
\centering
\begin{tikzpicture}[scale=0.85]
\draw(2,-1)--(0,0)--(0.5,-1);\node[below]at(0.5,-1){$\substack{3}$};
\draw(4,-2)--(2,-1)--(2.5,-2);\node[below]at(2.5,-2){$\substack{3\times2}$};
\draw(4,-2)--(3.875,-3)--(4.125,-3)--cycle;\node[below]at(4,-3){$\substack{\{3\}\times2^2\\2\text{ branches}}$};
\draw[ultra thick,-latex](0,0)--(-2,-1);\node[below]at(-2,-1){$\substack{3^2,3^3,\dotsc,3^{p-1}}$};
\end{tikzpicture}
\caption{A tree diagram using a ``power branch" notation}
\label{fig:highpowercondensed}
\end{minipage}
\end{figure}

We will use tree diagrams similar to Figure~\ref{fig:highpowercondensed} throughout the rest of this paper to denote various covering systems.

\section{Investigating Question~\ref{question:mainquestionsquarefree}}\label{sec:squarefree}

\iffalse%%%%%%%%%%%%%%%%%%
To begin our investigation for square-free covering systems of the integers, we first provide the following lemma, originally given by Filaseta and Harvey \cite{fh}.

\begin{lemma}\label{lem:shifted}
Let $d$ be an integer.  Given a covering system $\mathcal{C}=\{\mod{r_j}{m_j}: 1\leq j\leq k\}$ of $S\subseteq\mathbb{Z}$, the set $\mathcal{C}_
d=\{r_j+d\pmod{m_j}:1\leq j\leq k\}$ is a covering of the set $S_d=\{s+d:s\in S\}$.
\end{lemma}
\fi%%%%%%%%%%%%%%%%%%

To begin our investigation for square-free covering systems of the integers, we first provide the following lemma, originally given by Hammer, Harrington, and Marotta \cite{hhm}.

\begin{lemma}\label{lem:shifted}
Let $p$ be a prime and let $r_1$ and $r_2$ be distinct integers with $0\leq r_1,r_2\leq p-1$. Suppose $\mathcal{C}$ is a covering system of the integers such that $\mod{r_1}{p}\in\mathcal{C}$ and $\mod{r_2}{p}\notin\mathcal{C}$. Then there exists a covering system $\mathcal{C}'$ with the same moduli as those of $\mathcal{C}$ such that $\mod{r_2}{p}\in\mathcal{C}'$ and $\mod{r_1}{p}\notin\mathcal{C}'$. Furthermore, if $\mod{r}{p}\in\mathcal{C}$ with $r\not\equiv\mod{r_1}{p}$, then $\mod{r}{p}\in\mathcal{C}'$.
\end{lemma}

\begin{theorem}\label{thm:3sqrfreeimpliesodd}
Let $p\geq3$ be a prime. If there exists a covering system of the integers such that all moduli are odd, square-free, and distinct except that $p$ is used exactly twice as a modulus, then there exists an odd covering of the integers.
\end{theorem}

\begin{proof}
Suppose that $\mathcal{C}_0$ is a covering system of the integers such that all moduli are odd, square-free, and distinct except that $p$ is used exactly twice as a modulus. By Lemma~\ref{lem:shifted}, we may assume that $\mod{0}{p}$ and $\mod{1}{p}$ are congruences in $\mathcal{C}_0$. Let $k$ be a nonnegative integer and let $\ell_\xi$ be nonnegative integers for each $2\leq\xi\leq p-1$ such that
\begin{align*}
\mathcal{C}_0&=\{\mod{0}{p},\mod{1}{p}\}\cup\{\mod{r_j}{m_j}:1\leq j\leq k\}\cup\\
&\hspace{20pt}\bigcup_{\xi=2}^{p-1}\{\mod{r_{\xi,j}}{pm_{\xi,j}}:1\leq j\leq\ell_\xi\},
\end{align*}
where $p\nmid m_j$ for all $1\leq j\leq k$, and $p\nmid m_{\xi,j}$ and $r_{\xi,j}\equiv\mod{\xi}{p}$ for all $1\leq j\leq k_\xi$ and $2\leq\xi\leq p-1$.

We claim that for each $2\leq\xi\leq p-1$, the set of congruences 
$$\mathcal{C}_\xi=\{\mod{r_j}{m_j}:1\leq j\leq k\}\cup\{\mod{r_{\xi,j}}{m_{\xi,j}}:1\leq j\leq\ell_\xi\}$$
forms a covering system of the integers. To see this, consider any integer $r$. Let
$$m=\lcm\big(\{m_j:1\leq j\leq k\}\cup\{m_{\xi,j}:1\leq j\leq\ell_\xi\}\big)$$
and let $n$ be an integer that satisfies the congruence system
$$\begin{cases}
n\equiv\mod{\xi}{p},\\
n\equiv\mod{r}{m}.
\end{cases}$$
Since $\mathcal{C}_0$ is a covering system of the integers, $n$ must satisfy one of the congruences in
$$
\{\mod{r_j}{m_j}:1\leq j\leq k\}\cup\{\mod{r_{\xi,j}}{pm_{\xi,j}}:1\leq j\leq\ell_\xi\}.$$
As a result, if $n\equiv\mod{r_{\tilde{j}}}{m_{\tilde{j}}}$ for some $1\leq\tilde{j}\leq k$, then
$$r\equiv n\equiv\mod{r_{\tilde{j}}}{m_{\tilde{j}}};$$
if $n\equiv\mod{r_{\xi,\tilde{j}}}{pm_{\xi,\tilde{j}}}$ for some $1\leq\tilde{j}\leq\ell_\xi$, then
$$r\equiv n\equiv\mod{r_{\xi,\tilde{j}}}{m_{\xi,\tilde{j}}}.$$
This completes the proof of our claim.

Let $q$ be an odd prime such that $q\nmid pm$. For each $1\leq i\leq q-2$, $2\leq\xi\leq p-2$, and $1\leq j\leq\ell_\xi$, let $r_{\xi,i,j}$ be an integer that satisfies the congruence system
$$\begin{cases}
r_{\xi,i,j}\equiv\mod{\xi\cdot p^i}{p^{i+1}},\\
r_{\xi,i,j}\equiv\mod{r_{\xi,j}}{m_{\xi,j}}.
\end{cases}$$
Moreover, for each $0\leq i\leq q-1$, let $s_i$ be an integer that satisfies the congruence system
$$\begin{cases}
s_i\equiv\mod{0}{p^i},\\
s_i\equiv\mod{i}{q}.
\end{cases}$$
Now, we are ready to build an odd covering system based on the existence of $\mathcal{C}_0$. Consider
\begin{align*}
\mathcal{C}&=\{\mod{p^i}{p^{i+1}}:0\leq i\leq q-2\}\cup\{\mod{s_i}{p^iq}:0\leq i\leq q-1\}\\
&\hspace{20pt}\cup\{\mod{r_j}{m_j}:1\leq j\leq k\}\cup\bigcup_{\xi=2}^{p-1}\{\mod{r_{\xi,i,j}}{p^{i+1}m_{\xi,j}}:0\leq i\leq q-2,1\leq j\leq\ell_\xi\}.
\end{align*}

It is easy to see that all moduli of $\mathcal{C}$ are odd and distinct. To show that $\mathcal{C}$ is a covering system of the integers, let $N$ be an arbitrary integer, and let $i_0$ be the maximum integer such that $p^{i_0}\mid N$. If $0\leq i_0\leq q-2$, then either $N\equiv\mod{p^{i_0}}{p^{i_0+1}}$, which is trivially covered by $\mathcal{C}$, or $N\equiv\mod{\xi\cdot p^{i_0}}{p^{i_0+1}}$ for some $2\leq\xi\leq p-1$, which is covered by 
$$\{\mod{r_j}{m_j}:1\leq j\leq k\}\cup\{\mod{r_{\xi,i_0,j}}{p^{i_0+1}m_{\xi,j}}:1\leq j\leq\ell_\xi\}$$
since $C_\xi$ is a covering system of the integers. Lastly, if $i_0\geq q-1$, then $N$ is covered by $\{\mod{s_i}{p^iq}:0\leq i\leq q-1\}$.
\end{proof}

\begin{remark}
The proof of Theorem~\ref{thm:3sqrfreeimpliesodd} can be summarized as a transformation between two tree diagrams. Consider a tree diagram for $\mathcal{C}_0$ as shown in Figure~\ref{fig:C0}. In this tree diagram, the root is a $p$-node, and for each $i\in\{2,3,\dotsc,p-1\}$, $T_i$ denotes the subtree below the node that corresponds to the congruence $\mod{i}{p}$. If we replace the leftmost branch by a power branch $p^2,p^3,\dotsc,p^{q-1}$, as shown in Figure~\ref{fig:C}, then the resultant tree diagram represents a covering system $\mathcal{C}$ with odd and distinct moduli that are greater than $1$.
\begin{figure}[H]
\centering
\begin{minipage}{0.48\textwidth}
\centering
\begin{tikzpicture}[scale=1.2]
\foreach[count=\j]\i in{-2,-1,0,1,3}{\draw(0,0)--(\i,-1);}\draw[very thick,loosely dotted](1.5,-1)--(2.5,-1);
\foreach[count=\j]\i in{-2,-1}{\node[below]at(\i,-1){$\substack{p}$};}\node[below]at(0,-1){$\substack{T_2}$};\node[below]at(1,-1){$\substack{T_3}$};\node[below]at(3,-1){$\substack{T_{p-1}}$};
\end{tikzpicture}
\caption{The root of a tree diagram for $\mathcal{C}_0$}
\label{fig:C0}
\end{minipage}
\begin{minipage}{0.48\textwidth}
\centering
\begin{tikzpicture}[scale=1.2]
\draw[ultra thick,-latex](0,0)--(-2,-1);\foreach[count=\j]\i in{-1,0,1,3}{\draw(0,0)--(\i,-1);}\draw[very thick,loosely dotted](1.5,-1)--(2.5,-1);
\node[below]at(-2.1,-1){$\substack{p^2,p^3,\dotsc,p^{q-1}}$};
\node[below]at(-1,-1){$\substack{p}$};\node[below]at(0,-1){$\substack{T_2}$};\node[below]at(1,-1){$\substack{T_3}$};\node[below]at(3,-1){$\substack{T_{p-1}}$};
\end{tikzpicture}
\caption{The root of a tree diagram for $\mathcal{C}$}
\label{fig:C}
\end{minipage}
\end{figure}
\end{remark}

We note that Theorem~\ref{thm:3sqrfreeimpliesodd} can be restated as follows: if $\tau_p\leq2$ for any prime $p\geq3$, then there is an odd covering system of the integers. Hence, to investigate the existence of an odd covering system, one possible approach is to give bounds on $\tau_p$ for an odd prime $p$. In the following, we show that $\tau_7\leq6$.

\begin{theorem}\label{thm:tau7}
There exists a covering system of the integers such that all moduli are odd, square-free, and distinct except that $7$ is used exactly six times as a modulus.  
\end{theorem}
\begin{proof}
Figure~\ref{fig:six7s} shows a tree diagram for a covering system of the integers such that all moduli are odd, square-free, and distinct except that $7$ is used exactly six times as a modulus. In the tree diagram, note that at the wedge $\{7,3,5,11\}\times13$, there are $16$ available moduli according to the discussions in Section~\ref{sec:tree}, but there are only $13$ branches since it is a $13$-node. Thus, we are going to choose the $13$ smallest moduli given by the product of subsets of $\{7,3,5,11\}$ with $13$ among the child branches. Similar situations occurs at a few other wedges, as indicated in the figure.
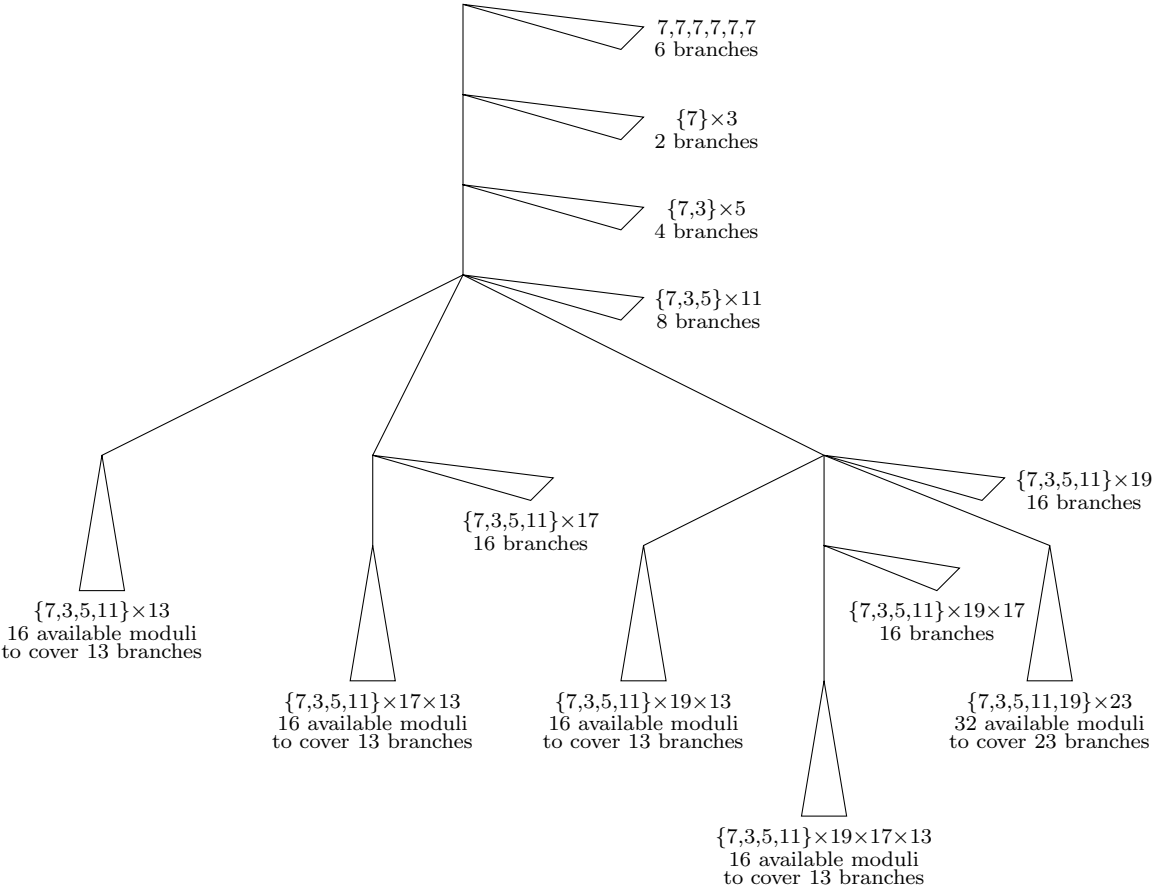
\begin{figure}[H]
\centering
\begin{tikzpicture}[scale=1.2]
\draw(0,0)--(2,-0.25)--(1.75,-0.5)--(0,0)--(0,-1)--(2,-1.25)--(1.75,-1.5)--(0,-1)--(0,-2)--(2,-2.25)--(1.75,-2.5)--(0,-2)--(0,-3)--(2,-3.25)--(1.75,-3.5)--(0,-3)--(-4,-5)--(-4.25,-6.5)--(-3.75,-6.5)--(-4,-5);
\draw(0,-3)--(-1,-5)--(1,-5.25)--(0.75,-5.5)--(-1,-5)--(-1,-6)--(-1.25,-7.5)--(-0.75,-7.5)--(-1,-6);
\draw(0,-3)--(4,-5)--(6,-5.25)--(5.75,-5.5)--(4,-5)--(2,-6)--(1.75,-7.5)--(2.25,-7.5)--(2,-6);
\draw(4,-5)--(4,-6)--(5.5,-6.25)--(5.25,-6.5)--(4,-6)--(4,-7.5)--(3.75,-9)--(4.25,-9)--(4,-7.5);
\draw(4,-5)--(6.5,-6)--(6.25,-7.5)--(6.75,-7.5)--(6.5,-6);
\node[right]at(2,-0.375){$\substack{7,7,7,7,7,7\\6\text{ branches}}$};
\node[right]at(2,-1.375){$\substack{\{7\}\times3\\2\text{ branches}}$};
\node[right]at(2,-2.375){$\substack{\{7,3\}\times5\\4\text{ branches}}$};
\node[right]at(2,-3.375){$\substack{\{7,3,5\}\times11\\8\text{ branches}}$};
\node[below]at(-4,-6.5){$\substack{\{7,3,5,11\}\times13\\16\text{ available moduli}\\\text{to cover }13\text{ branches}}$};
\node[right]at(6,-5.375){$\substack{\{7,3,5,11\}\times19\\16\text{ branches}}$};
\node[below]at(-1,-7.5){$\substack{\{7,3,5,11\}\times17\times13\\16\text{ available moduli}\\\text{to cover }13\text{ branches}}$};
\node[below]at(0.75,-5.5){$\substack{\{7,3,5,11\}\times17\\16\text{ branches}}$};
\node[below]at(2,-7.5){$\substack{\{7,3,5,11\}\times19\times13\\16\text{ available moduli}\\\text{to cover }13\text{ branches}}$};
\node[below]at(5.25,-6.5){$\substack{\{7,3,5,11\}\times19\times17\\16\text{ branches}}$};
\node[below]at(4,-9){$\substack{\{7,3,5,11\}\times19\times17\times13\\16\text{ available moduli}\\\text{to cover }13\text{ branches}}$};
\node[below]at(6.5,-7.5){$\substack{\{7,3,5,11,19\}\times23\\32\text{ available moduli}\\\text{to cover }23\text{ branches}}$};
\end{tikzpicture}
\caption{Square-free odd covering with $7$ used exactly six times as a modulus}
\label{fig:six7s}
\end{figure}
\end{proof}

\section{Investigating Question~\ref{question:mainquestion}}\label{sec:main}

In this section, we establish upper bounds for $t_p$ for all odd primes $p$.  We begin by showing that $t_7\leq 4.$

\begin{theorem}\label{thm:t7}
There exists a covering system of the integers such that all moduli are odd and distinct except that $7$ is used exactly four times as a modulus.
\end{theorem}
\begin{proof}
The tree diagram of this covering system is given by Figure~\ref{fig:four7s}, where subtrees $T_1$ to $T_6$ are given in Figures~\ref{fig:four7sT1} to \ref{fig:four7sT6}, respectively. Here, $q>19$ is a prime.

Please take special note in the power branch $7^3,7^4,\dotsc,7^{q-1}$ in Figure~\ref{fig:four7sT2}. Here, the repeated subtrees are obtained by substituting every factor of $7^2$ (but not $7$) in the first repeated subtree with $7^i$ for $3\leq i\leq q-1$.

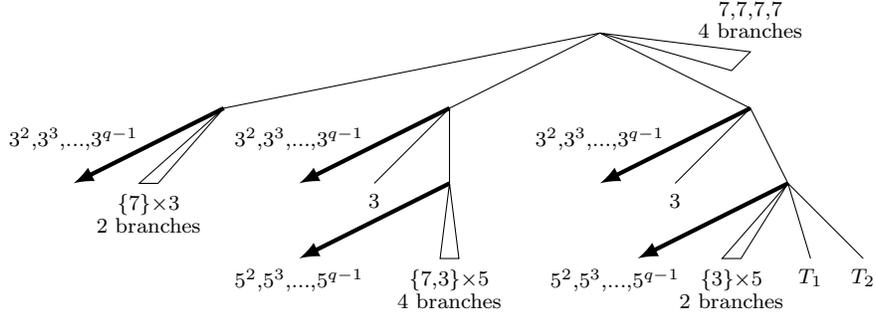
\begin{figure}[H]
\centering
\begin{tikzpicture}[scale=1]
\draw(0,0)--(2,-0.25)--(1.75,-0.5)--(0,0);
\node[above]at(2,-0.2){$\substack{7,7,7,7\\4\text{ branches}}$};
\draw(0,0)--(-5,-1)--(-6.125,-2)--(-5.875,-2)--(-5,-1);\draw[ultra thick,-latex](-5,-1)--(-7,-2);
\node[above]at(-7,-1.7){$\substack{3^2,3^3,\dotsc,3^{q-1}}$};\node[below]at(-6,-2){$\substack{\{7\}\times3\\2\text{ branches}}$};

\draw(0,0)--(-2,-1)--(-3,-2);\draw[ultra thick,-latex](-2,-1)--(-4,-2);
\node[above]at(-4,-1.7){$\substack{3^2,3^3,\dotsc,3^{q-1}}$};\node[below]at(-3,-2){$\substack{3}$};
\draw(-2,-1)--(-2,-2)--(-2.125,-3)--(-1.875,-3)--(-2,-2);\draw[ultra thick,-latex](-2,-2)--(-4,-3);
\node[below]at(-4,-3){$\substack{5^2,5^3,\dotsc,5^{q-1}}$};\node[below]at(-2,-3){$\substack{\{7,3\}\times5\\4\text{ branches}}$};

\draw(0,0)--(2,-1)--(1,-2);\draw[ultra thick,-latex](2,-1)--(0,-2);
\node[above]at(0,-1.7){$\substack{3^2,3^3,\dotsc,3^{q-1}}$};\node[below]at(1,-2){$\substack{3}$};
\draw(2,-1)--(2.5,-2)--(1.625,-3)--(1.875,-3)--(2.5,-2);\draw[ultra thick,-latex](2.5,-2)--(0.5,-3);
\node[below]at(0.2,-3){$\substack{5^2,5^3,\dotsc,5^{q-1}}$};\node[below]at(1.75,-3){$\substack{\{3\}\times5\\2\text{ branches}}$};
\draw(2.5,-2)--(2.8,-3);\draw(2.5,-2)--(3.5,-3);\node[below]at(2.8,-3){$\substack{T_1}$};\node[below]at(3.5,-3){$\substack{T_2}$};
\end{tikzpicture}
\caption{Odd covering with $7$ used exactly four times as a modulus}
\label{fig:four7s}
\end{figure}

\begin{figure}[H]
\centering
\begin{tikzpicture}[scale=0.98]
\draw[ultra thick,-latex](0,0)--(-8,-1);\node[below]at(-8.3,-1){$\substack{11^2,11^3,\dotsc,11^{q-1}}$};
\draw(0,0)--(-6.5,-1)--(-6,-1)--(0,0);\node[below]at(-6.25,-1){$\substack{\{7,3\}\times11\\4\text{ branches}}$};
\draw(0,0)--(-4.5,-1)--(-4,-1)--(0,0);\node[below]at(-4.25,-1){$\substack{\{7,3\}\times5\times11\\4\text{ branches}}$};
\draw(0,0)--(-2,-1);\draw[ultra thick,-latex](-2,-1)--(-4,-2);\node[below]at(-4.7,-2){$\substack{17^2,17^3,\dotsc,17^{q-1}}$};
\draw(-2,-1)--(-2.625,-2)--(-2.375,-2)--(-2,-1);\node[below]at(-2.5,-2){$\substack{\{7,3,5\}\times17\\8\text{ branches}}$};
\draw(-2,-1)--(-0.625,-2)--(-0.375,-2)--(-2,-1);\node[below]at(-0.5,-2){$\substack{\{7,3,5\}\times11\times17\\8\text{ branches}}$};
\draw(0,0)--(3,-1);\draw[ultra thick,-latex](3,-1)--(1,-2);\node[above]at(0.9,-1.7){$\substack{19^2,19^3,\dotsc,19^{q-1}}$};
\draw(3,-1)--(2.375,-2)--(2.625,-2)--(3,-1);\node[below]at(2.5,-2){$\substack{\{7,3,5\}\times19\\8\text{ branches}}$};
\draw(3,-1)--(4.375,-2)--(4.625,-2)--(3,-1);\node[below]at(4.5,-2){$\substack{\{7,3,5\}\times11\times19\\8\text{ branches}}$};
\draw(3,-1)--(6,-2);\node[below]at(6,-2){$\substack{T_3}$};
\draw(3,-1)--(7,-2);\node[below]at(7,-2){$\substack{T_4}$};
\end{tikzpicture}
\caption{$T_1$ in Figure~\ref{fig:four7s}}
\label{fig:four7sT1}
\end{figure}

\begin{figure}[H]
\centering
\begin{minipage}{0.48\textwidth}
\centering
\begin{tikzpicture}[scale=1]
\draw[ultra thick,-latex](0,0)--(-2,-1);\node[above]at(-2.4,-0.8){$\substack{17^2,17^3,\dotsc,17^{q-1}}$};
\draw(0,0)--(-1.125,-1)--(-0.875,-1)--(0,0);\node[below]at(-1,-1){$\substack{\{7,3,5\}\times17\\8\text{ branches}}$};
\draw(0,0)--(0.575,-1)--(0.825,-1)--(0,0);\node[below]at(0.7,-1){$\substack{\{7,3,5\}\\\times19\times17\\8\text{ branches}}$};
\end{tikzpicture}
\caption{$T_3$ in Figure~\ref{fig:four7sT1}}
\label{fig:four7sT3}
\end{minipage}
\begin{minipage}{0.48\textwidth}
\centering
\begin{tikzpicture}[scale=0.92]
\draw[ultra thick,-latex](0,0)--(-2,-1);\node[above]at(-2.4,-0.8){$\substack{17^2,17^3,\dotsc,17^{q-1}}$};
\draw(0,0)--(-1.125,-1)--(-0.875,-1)--(0,0);\node[below]at(-1,-1){$\substack{\{7,3,5\}\times17\\8\text{ branches}}$};
\draw(0,0)--(0.875,-1)--(1.125,-1)--(0,0);\node[below]at(1,-1){$\substack{\{7,3,5\}\\\times11\times19\times17\\8\text{ branches}}$};
\end{tikzpicture}
\caption{$T_4$ in Figure~\ref{fig:four7sT1}}
\label{fig:four7sT4}
\end{minipage}
\end{figure}

\begin{figure}[H]
\centering
\begin{tikzpicture}[scale=0.98]
\draw[ultra thick,-latex](0,0)--(-8,-1);\node[below]at(-8.3,-1){$\substack{7^3,7^4,\dotsc,7^{q-1}}$};
\draw(0,0)--(-6.5,-1)--(-6,-1)--(0,0);\node[below]at(-6.25,-1){$\substack{\{3,5\}\times7^2\\4\text{ branches}}$};

\draw(0,0)--(-4,-1);\draw[ultra thick,-latex](-4,-1)--(-6,-2);\node[below]at(-6.6,-2){$\substack{13^2,13^3,\dotsc,13^{q-1}}$};
\draw(-4,-1)--(-4.625,-2)--(-4.375,-2)--(-4,-1);\node[below]at(-4.5,-2){$\substack{\{7,3,5\}\times13\\8\text{ branches}}$};
\draw(-4,-1)--(-2.625,-2)--(-2.375,-2)--(-4,-1);\node[below]at(-2.5,-2){$\substack{\{3,5\}\times7^2\times13\\4\text{ branches}}$};

\draw(0,0)--(1,-1);\draw[ultra thick,-latex](1,-1)--(-1,-2);\node[above]at(-1.2,-1.7){$\substack{11^2,11^3,\dotsc,11^{q-1}}$};
\draw(1,-1)--(0.375,-2)--(0.625,-2)--(1,-1);\node[below]at(0.5,-2){$\substack{\{7,3\}\times11\\4\text{ branches}}$};
\draw(1,-1)--(2.375,-2)--(2.625,-2)--(1,-1);\node[below]at(2.5,-2){$\substack{\{3,5\}\times7^2\times11\\4\text{ branches}}$};
\draw(1,-1)--(4,-2);\node[below]at(4,-2){$\substack{T_5}$};
\draw(1,-1)--(4.7,-2);\node[below]at(4.7,-2){$\substack{T_6}$};

%\draw(0,0)--(-6.5,-1)--(-6,-1)--(0,0);\node[below]at(-6.5,-1){$\substack{\{3,5\}\times7^2\\4\text{ branches}}$};

%\draw(0,0)--(-4,-1);\draw[ultra thick,-latex](-4,-1)--(-6,-2);\node at(-7.3,-2){$\substack{13^2,13^3,\dotsc,13^{q-1}}$};
%\draw(-4,-1)--(-5.125,-2)--(-4.875,-2)--(-4,-1);\node[below]at(-5,-2){$\substack{\{7,3,5\}\times13\\8\text{ branches}}$};
%\draw(-4,-1)--(-3.625,-2)--(-3.375,-2)--(-4,-1);\node[below]at(-3.5,-2){$\substack{\{3,5\}\\\times7^2\times13\\4\text{ branches}}$};

%\draw(0,0)--(1,-1);\draw[ultra thick,-latex](1,-1)--(-2.5,-2);\node[above]at(-2.1,-1.7){$\substack{11^2,11^3,\dotsc,11^{q-1}}$};
%\draw(1,-1)--(-1.525,-2)--(-1.275,-2)--(1,-1);\node[below]at(-1.4,-2){$\substack{\{7,3\}\times11\\4\text{ branches}}$};
%\draw(1,-1)--(0.375,-2)--(0.625,-2)--(1,-1);\node[below]at(0.5,-2){$\substack{\{3,5\}\times7^2\times11\\4\text{ branches}}$};

%\draw(1,-1)--(3,-2);\draw[ultra thick,-latex](3,-2)--(1,-3);
\end{tikzpicture}
\caption{$T_2$ in Figure~\ref{fig:four7s}}
\label{fig:four7sT2}
\end{figure}

\begin{figure}[H]
\centering
\begin{minipage}{0.48\textwidth}
\centering
\begin{tikzpicture}[scale=1]
\draw[ultra thick,-latex](0,0)--(-2,-1);\node[above]at(-2.4,-0.8){$\substack{13^2,13^3,\dotsc,13^{q-1}}$};
\draw(0,0)--(-0.625,-1)--(-0.375,-1)--(0,0);\node[below]at(-0.5,-1){$\substack{\{7,3,5\}\times13\\8\text{ branches}}$};
\draw(0,0)--(1.375,-1)--(1.625,-1)--(0,0);\node[below]at(1.5,-1){$\substack{\{7,3\}\times11\times13\\4\text{ branches}}$};
\end{tikzpicture}
\caption{$T_5$ in Figure~\ref{fig:four7sT2}}
\label{fig:four7sT5}
\end{minipage}
\begin{minipage}{0.48\textwidth}
\centering
\begin{tikzpicture}[scale=0.92]
\draw[ultra thick,-latex](0,0)--(-2,-1);\node[above]at(-2.4,-0.8){$\substack{13^2,13^3,\dotsc,13^{q-1}}$};
\draw(0,0)--(-0.625,-1)--(-0.375,-1)--(0,0);\node[below]at(-0.5,-1){$\substack{\{7,3,5\}\times13\\8\text{ branches}}$};
\draw(0,0)--(1.875,-1)--(2.125,-1)--(0,0);\node[below]at(2,-1){$\substack{\{7,3\}\times5\times11\times13\\4\text{ branches}}$};
\end{tikzpicture}
\caption{$T_6$ in Figure~\ref{fig:four7sT2}}
\label{fig:four7sT6}
\end{minipage}
\end{figure}
\end{proof}

Next, we show that $t_{11}\leq 7$.

\begin{theorem}\label{thm:t11}
There exists a covering system of the integers such that all moduli are odd and distinct except that $11$ is used exactly seven times as a modulus.
\end{theorem}
\begin{proof}
The tree diagram of this covering system is given by Figure~\ref{fig:seven11s}, where subtrees $T_1$ to $T_4$ are given in Figures~\ref{fig:seven11sT1} to \ref{fig:seven11sT4}, respectively.  Here, $q>19$ is a prime.
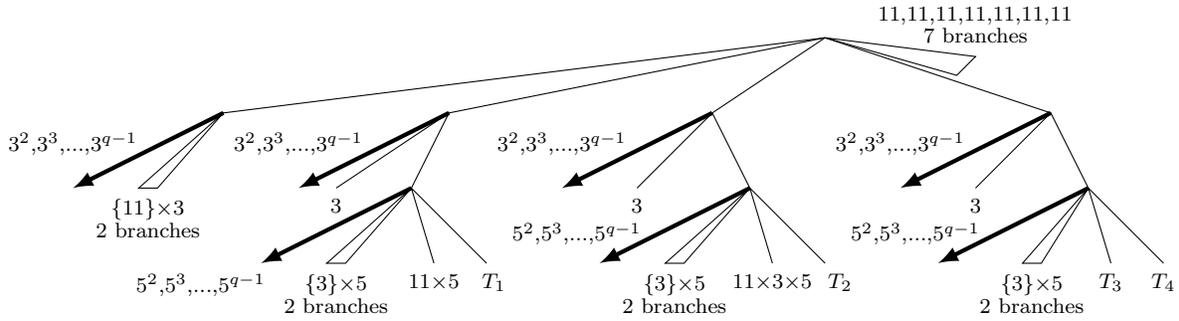
\begin{figure}[H]
\centering
\begin{tikzpicture}[scale=1]
\draw(0,0)--(2,-0.25)--(1.75,-0.5)--(0,0);
\node[above]at(2,-0.2){$\substack{11,11,11,11,11,11,11\\7\text{ branches}}$};
\draw(0,0)--(-8,-1)--(-9.125,-2)--(-8.875,-2)--(-8,-1);\draw[ultra thick,-latex](-8,-1)--(-10,-2);
\node[above]at(-10,-1.7){$\substack{3^2,3^3,\dotsc,3^{q-1}}$};\node[below]at(-9,-2){$\substack{\{11\}\times3\\2\text{ branches}}$};

\draw(0,0)--(-5,-1)--(-6.5,-2);\draw[ultra thick,-latex](-5,-1)--(-7,-2);
\node[above]at(-7,-1.7){$\substack{3^2,3^3,\dotsc,3^{q-1}}$};\node[below]at(-6.5,-2){$\substack{3}$};
\draw(-5,-1)--(-5.5,-2)--(-6.625,-3)--(-6.375,-3)--(-5.5,-2);\draw[ultra thick,-latex](-5.5,-2)--(-7.5,-3);
\node[below]at(-8.3,-3){$\substack{5^2,5^3,\dotsc,5^{q-1}}$};\node[below]at(-6.5,-3){$\substack{\{3\}\times5\\2\text{ branches}}$};
\draw(-5.5,-2)--(-5.2,-3);\draw(-5.5,-2)--(-4.5,-3);\node[below]at(-5.2,-3){$\substack{11\times5}$};\node[below]at(-4.4,-3){$\substack{T_1}$};

\draw(0,0)--(-1.5,-1)--(-2.5,-2);\draw[ultra thick,-latex](-1.5,-1)--(-3.5,-2);
\node[above]at(-3.5,-1.7){$\substack{3^2,3^3,\dotsc,3^{q-1}}$};\node[below]at(-2.5,-2){$\substack{3}$};
\draw(-1.5,-1)--(-1,-2)--(-2.125,-3)--(-1.875,-3)--(-1,-2);\draw[ultra thick,-latex](-1,-2)--(-3,-3);
\node[above]at(-3.3,-2.9){$\substack{5^2,5^3,\dotsc,5^{q-1}}$};\node[below]at(-2,-3){$\substack{\{3\}\times5\\2\text{ branches}}$};
\draw(-1,-2)--(-0.7,-3);\draw(-1,-2)--(0,-3);\node[below]at(-0.7,-3){$\substack{11\times3\times5}$};\node[below]at(0.2,-3){$\substack{T_2}$};

\draw(0,0)--(3,-1)--(2,-2);\draw[ultra thick,-latex](3,-1)--(1,-2);
\node[above]at(1,-1.7){$\substack{3^2,3^3,\dotsc,3^{q-1}}$};\node[below]at(2,-2){$\substack{3}$};
\draw(3,-1)--(3.5,-2)--(2.625,-3)--(2.875,-3)--(3.5,-2);\draw[ultra thick,-latex](3.5,-2)--(1.5,-3);
\node[above]at(1.2,-2.9){$\substack{5^2,5^3,\dotsc,5^{q-1}}$};\node[below]at(2.75,-3){$\substack{\{3\}\times5\\2\text{ branches}}$};
\draw(3.5,-2)--(3.8,-3);\draw(3.5,-2)--(4.5,-3);\node[below]at(3.8,-3){$\substack{T_3}$};\node[below]at(4.5,-3){$\substack{T_4}$};
\end{tikzpicture}
\caption{Odd covering with $11$ used exactly seven times as a modulus}
\label{fig:seven11s}
\end{figure}

\begin{figure}[H]
\centering
\begin{tikzpicture}[scale=0.98]
\draw[ultra thick,-latex](0,0)--(-8,-1);\node[below]at(-8.3,-1){$\substack{7^2,7^3,\dotsc,7^{q-1}}$};
\draw(0,0)--(-6.5,-1)--(-6,-1)--(0,0);\node[below]at(-6.25,-1){$\substack{\{3\}\times7\\2\text{ branches}}$};
\draw(0,0)--(-4.5,-1)--(-4,-1)--(0,0);\node[below]at(-4.25,-1){$\substack{\{3\}\times5\times7\\2\text{ branches}}$};
\draw(0,0)--(-2,-1);\draw[ultra thick,-latex](-2,-1)--(-4,-2);\node[below]at(-4.7,-2){$\substack{13^2,13^3,\dotsc,13^{q-1}}$};
\draw(-2,-1)--(-2.625,-2)--(-2.375,-2)--(-2,-1);\node[below]at(-2.5,-2){$\substack{\{11,3,5\}\times13\\8\text{ branches}}$};
\draw(-2,-1)--(-0.625,-2)--(-0.375,-2)--(-2,-1);\node[below]at(-0.5,-2){$\substack{\{11,3\}\times7\times13\\4\text{ branches}}$};
\draw(0,0)--(4.7,-1);\draw[ultra thick,-latex](4.7,-1)--(2.7,-2);\node[below]at(2,-2){$\substack{13^2,13^3,\dotsc,13^{q-1}}$};
\draw(4.7,-1)--(4.075,-2)--(4.325,-2)--(4.7,-1);\node[below]at(4.2,-2){$\substack{\{11,3,5\}\times13\\8\text{ branches}}$};
\draw(4.7,-1)--(6.075,-2)--(6.325,-2)--(4.7,-1);\node[below]at(6.2,-2){$\substack{\{11,3\}\times5\times7\times13\\4\text{ branches}}$};
\end{tikzpicture}
\caption{$T_1$ in Figure~\ref{fig:seven11s}}
\label{fig:seven11sT1}
\end{figure}

\begin{figure}[H]
\centering
\begin{tikzpicture}[scale=0.96]
\draw[ultra thick,-latex](0,0)--(-8,-1);\node[below]at(-8.3,-1){$\substack{7^2,7^3,\dotsc,7^{q-1}}$};
\draw(0,0)--(-6.5,-1)--(-6,-1)--(0,0);\node[below]at(-6.25,-1){$\substack{\{3\}\times7\\2\text{ branches}}$};
\draw(0,0)--(-4.5,-1)--(-4,-1)--(0,0);\node[below]at(-4.25,-1){$\substack{\{3\}\times5\times7\\2\text{ branches}}$};
\draw(0,0)--(-2,-1);\draw[ultra thick,-latex](-2,-1)--(-4,-2);\node[below]at(-4.4,-2){$\substack{17^2,17^3,\dotsc,17^{q-1}}$};
\draw(-2,-1)--(-2.125,-2)--(-1.875,-2)--(-2,-1);\node[below]at(-2,-2){$\substack{\{11,3,5,7\}\times17\\16\text{ branches}}$};

\draw(0,0)--(4,-1);\draw[ultra thick,-latex](4,-1)--(0.5,-2);\node[above]at(0.3,-1.9){$\substack{19^2,19^3,\dotsc,19^{q-1}}$};
\draw(4,-1)--(6,-1.125)--(5.75,-1.25)--(4,-1);\node[above]at(6,-1.075){$\substack{\{11,3,5,7\}\times19\\16\text{ branches}}$};
\draw(4,-1)--(2,-2);\draw[ultra thick,-latex](2,-2)--(0,-3);\node[below]at(-0.7,-3){$\substack{13^2,13^3,\dotsc,13^{q-1}}$};
\draw(2,-2)--(1.875,-3)--(2.125,-3)--(2,-2);\node[below]at(2,-3){$\substack{\{11,3,5,7\}\times19\times13\\16\text{ available moduli}\\\text{to cover }12\text{ branches}}$};
\draw(4,-1)--(6.5,-2);\draw[ultra thick,-latex](6.5,-2)--(4.5,-3);\node[above]at(4.1,-2.8){$\substack{17^2,17^3,\dotsc,17^{q-1}}$};
\draw(6.5,-2)--(6.375,-3)--(6.625,-3)--(6.5,-2);\node[below]at(6.5,-3){$\substack{\{11,3,5,7\}\times19\times17\\16\text{ branches}}$};
\end{tikzpicture}
\caption{$T_2$ in Figure~\ref{fig:seven11s}}
\label{fig:seven11sT2}
\end{figure}

\begin{figure}[H]
\centering
\begin{minipage}{0.48\textwidth}
\centering
\begin{tikzpicture}[scale=1]
\draw[ultra thick,-latex](0,0)--(-4,-1);\node[above]at(-4,-0.8){$\substack{7^2,7^3,\dotsc,7^{q-1}}$};
\draw(0,0)--(-2.125,-1)--(-1.875,-1)--(0,0);\node[below]at(-2,-1){$\substack{\{3\}\times7\\2\text{ branches}}$};
\draw(0,0)--(-0.125,-1)--(0.125,-1)--(0,0);\node[below]at(0,-1){$\substack{\{3\}\times11\times5\times7\\2\text{ branches}}$};
\draw(0,0)--(1.875,-1)--(2.125,-1)--(0,0);\node[below]at(2,-1){$\substack{\{3\}\times11\times7\\2\text{ branches}}$};
\end{tikzpicture}
\caption{$T_3$ in Figure~\ref{fig:seven11s}}
\label{fig:seven11sT3}
\end{minipage}
\begin{minipage}{0.48\textwidth}
\centering
\begin{tikzpicture}[scale=1]
\draw[ultra thick,-latex](0,0)--(-4,-1);\node[above]at(-4,-0.8){$\substack{7^2,7^3,\dotsc,7^{q-1}}$};
\draw(0,0)--(-2.125,-1)--(-1.875,-1)--(0,0);\node[below]at(-2,-1){$\substack{\{3\}\times7\\2\text{ branches}}$};
\draw(0,0)--(-0.125,-1)--(0.125,-1)--(0,0);\node[below]at(0,-1){$\substack{\{3\}\times5\times7\\2\text{ branches}}$};
\draw(0,0)--(1.875,-1)--(2.125,-1)--(0,0);\node[below]at(2,-1){$\substack{\{3\}\times11\times7\\2\text{ branches}}$};
\end{tikzpicture}
\caption{$T_4$ in Figure~\ref{fig:seven11s}}
\label{fig:seven11sT4}
\end{minipage}
\end{figure}
\end{proof}

Lastly, we establish that $t_p\leq p-5$ for all primes $p\geq 23$.

\begin{theorem}\label{thm:t23}
Let $p\geq23$ be a prime. There exists a covering system of the integers such that all moduli are odd and distinct except that $p$ is used exactly $p-5$ times as a modulus.
\end{theorem}
\begin{proof}
The tree diagram of this covering system is given by Figure~\ref{fig:p-5ps}, where subtrees $T_1$ to $T_{11}$ are given in Figures~\ref{fig:p-5psT1} to \ref{fig:p-5psT11}, respectively. Here, $q>19$ and $q\neq p$ is a prime.
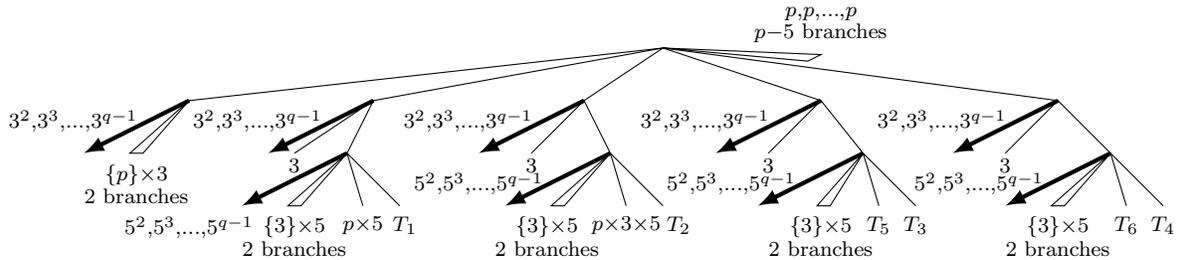
\begin{figure}[H]
\centering
\begin{tikzpicture}[scale=0.7]
\draw(0,0)--(3,-0.125)--(2.75,-0.25)--(0,0);
\node[above]at(3,-0.1){$\substack{p,p,\dotsc,p\\p-5\text{ branches}}$};
\draw(0,0)--(-9,-1)--(-10.125,-2)--(-9.875,-2)--(-9,-1);\draw[ultra thick,-latex](-9,-1)--(-11,-2);
\node[above]at(-11.2,-1.8){$\substack{3^2,3^3,\dotsc,3^{q-1}}$};\node[below]at(-10,-2){$\substack{\{p\}\times3\\2\text{ branches}}$};

\draw(0,0)--(-5.5,-1)--(-7,-2);\draw[ultra thick,-latex](-5.5,-1)--(-7.5,-2);
\node[above]at(-7.7,-1.8){$\substack{3^2,3^3,\dotsc,3^{q-1}}$};\node[below]at(-7,-1.9){$\substack{3}$};
\draw(-5.5,-1)--(-6,-2)--(-7.125,-3)--(-6.875,-3)--(-6,-2);\draw[ultra thick,-latex](-6,-2)--(-8,-3);
\node[below]at(-9,-3){$\substack{5^2,5^3,\dotsc,5^{q-1}}$};\node[below]at(-7,-3){$\substack{\{3\}\times5\\2\text{ branches}}$};
\draw(-6,-2)--(-5.7,-3);\draw(-6,-2)--(-5,-3);\node[below]at(-5.7,-3){$\substack{p\times5}$};\node[below]at(-4.9,-3){$\substack{T_1}$};

\draw(0,0)--(-1.5,-1)--(-2.5,-2);\draw[ultra thick,-latex](-1.5,-1)--(-3.5,-2);
\node[above]at(-3.7,-1.8){$\substack{3^2,3^3,\dotsc,3^{q-1}}$};\node[below]at(-2.5,-1.9){$\substack{3}$};
\draw(-1.5,-1)--(-1,-2)--(-2.125,-3)--(-1.875,-3)--(-1,-2);\draw[ultra thick,-latex](-1,-2)--(-3,-3);
\node[above]at(-3.5,-3){$\substack{5^2,5^3,\dotsc,5^{q-1}}$};\node[below]at(-2.2,-3){$\substack{\{3\}\times5\\2\text{ branches}}$};
\draw(-1,-2)--(-0.7,-3);\draw(-1,-2)--(0,-3);\node[below]at(-0.7,-3){$\substack{p\times3\times5}$};\node[below]at(0.3,-3){$\substack{T_2}$};

\draw(0,0)--(3,-1)--(2,-2);\draw[ultra thick,-latex](3,-1)--(1,-2);
\node[above]at(0.8,-1.8){$\substack{3^2,3^3,\dotsc,3^{q-1}}$};\node[below]at(2,-1.9){$\substack{3}$};
\draw(3,-1)--(3.8,-2)--(2.925,-3)--(3.175,-3)--(3.8,-2);\draw[ultra thick,-latex](3.8,-2)--(1.8,-3);
\node[above]at(1.3,-3){$\substack{5^2,5^3,\dotsc,5^{q-1}}$};\node[below]at(3,-3){$\substack{\{3\}\times5\\2\text{ branches}}$};
\draw(3.8,-2)--(4.1,-3);\draw(3.8,-2)--(4.8,-3);\node[below]at(4.1,-3){$\substack{T_5}$};\node[below]at(4.8,-3){$\substack{T_3}$};

\draw(0,0)--(7.5,-1)--(6.5,-2);\draw[ultra thick,-latex](7.5,-1)--(5.5,-2);
\node[above]at(5.3,-1.8){$\substack{3^2,3^3,\dotsc,3^{q-1}}$};\node[below]at(6.5,-1.9){$\substack{3}$};
\draw(7.5,-1)--(8.5,-2)--(7.375,-3)--(7.625,-3)--(8.5,-2);\draw[ultra thick,-latex](8.5,-2)--(6.5,-3);
\node[above]at(6,-3){$\substack{5^2,5^3,\dotsc,5^{q-1}}$};\node[below]at(7.5,-3){$\substack{\{3\}\times5\\2\text{ branches}}$};
\draw(8.5,-2)--(8.8,-3);\draw(8.5,-2)--(9.5,-3);\node[below]at(8.8,-3){$\substack{T_6}$};\node[below]at(9.5,-3){$\substack{T_4}$};
\end{tikzpicture}
\caption{Odd covering with $p$ ($p\geq23$) used exactly $p-5$ times as a modulus}
\label{fig:p-5ps}
\end{figure}

\begin{figure}[H]
\centering
\begin{minipage}{0.48\textwidth}
\centering
\begin{tikzpicture}[scale=0.92]
\draw[ultra thick,-latex](0,0)--(-4,-1);\node[above]at(-4,-0.8){$\substack{7^2,7^3,\dotsc,7^{q-1}}$};
\draw(0,0)--(-2.75,-1)--(-2.25,-1)--(0,0);\node[below]at(-2.5,-1){$\substack{\{3,5\}\times7\\4\text{ branches}}$};
\draw(0,0)--(-1,-1);\node[below]at(-1,-1){$\substack{p\times5\times7}$};
\draw(0,0)--(1,-1)--(-1.125,-2)--(-0.875,-2)--(1,-1)--(0.625,-2)--(0.875,-2)--(1,-1)--(2.375,-2)--(2.625,-2)--(1,-1);\draw[ultra thick,-latex](1,-1)--(-2.5,-2);\node[below]at(-3.3,-2){$\substack{11^2,11^3,\dotsc,11^{q-1}}$};
\node[below]at(-1,-2){$\substack{\{3,7\}\times5\times11\\4\text{ branches}}$};\node[below]at(0.75,-2){$\substack{\{3,7\}\times11\\4\text{ branches}}$};\node[below]at(2.5,-2){$\substack{\{3\}\times p\times5\times11\\2\text{ branches}}$};
\end{tikzpicture}
\caption{$T_1$ in Figure~\ref{fig:p-5ps}}
\label{fig:p-5psT1}
\end{minipage}
\begin{minipage}{0.48\textwidth}
\centering
\begin{tikzpicture}[scale=0.92]
\draw[ultra thick,-latex](0,0)--(-4,-1);\node[above]at(-4,-0.8){$\substack{7^2,7^3,\dotsc,7^{q-1}}$};
\draw(0,0)--(-2.75,-1)--(-2.25,-1)--(0,0);\node[below]at(-2.5,-1){$\substack{\{3,5\}\times7\\4\text{ branches}}$};
\draw(0,0)--(-1,-1);\node[below]at(-1,-1){$\substack{p\times3\times5\times7}$};
\draw(0,0)--(1,-1)--(-1.125,-2)--(-0.875,-2)--(1,-1)--(0.625,-2)--(0.875,-2)--(1,-1)--(2.375,-2)--(2.625,-2)--(1,-1);\draw[ultra thick,-latex](1,-1)--(-2.5,-2);\node[below]at(-3.3,-2){$\substack{11^2,11^3,\dotsc,11^{q-1}}$};
\node[below]at(-1,-2){$\substack{\{3,7\}\times5\times11\\4\text{ branches}}$};\node[below]at(0.75,-2){$\substack{\{3,7\}\times11\\4\text{ branches}}$};\node[below]at(2.7,-2){$\substack{\{3\}\times p\times5\times7\times11\\2\text{ branches}}$};
\end{tikzpicture}
\caption{$T_2$ in Figure~\ref{fig:p-5ps}}
\label{fig:p-5psT2}
\end{minipage}
\end{figure}

\begin{figure}[H]
\centering
\begin{minipage}{0.48\textwidth}
\centering
\begin{tikzpicture}[scale=0.92]
\draw[ultra thick,-latex](0,0)--(-4,-1);\node[above]at(-4,-0.8){$\substack{7^2,7^3,\dotsc,7^{q-1}}$};
\draw(0,0)--(-2.75,-1)--(-2.25,-1)--(0,0);\node[below]at(-2.5,-1){$\substack{\{3,5\}\times7\\4\text{ branches}}$};
\draw(0,0)--(-1,-1);\node[below]at(-1,-1){$\substack{p\times7}$};
\draw(0,0)--(1,-1)--(-1.125,-2)--(-0.875,-2)--(1,-1)--(0.625,-2)--(0.875,-2)--(1,-1)--(2.375,-2)--(2.625,-2)--(1,-1);\draw[ultra thick,-latex](1,-1)--(-2.5,-2);\node[below]at(-3.3,-2){$\substack{11^2,11^3,\dotsc,11^{q-1}}$};
\node[below]at(-1,-2){$\substack{\{3,7\}\times5\times11\\4\text{ branches}}$};\node[below]at(0.75,-2){$\substack{\{3,7\}\times11\\4\text{ branches}}$};\node[below]at(2.5,-2){$\substack{\{3\}\times p\times11\\2\text{ branches}}$};
\end{tikzpicture}
\caption{$T_3$ in Figure~\ref{fig:p-5ps}}
\label{fig:p-5psT3}
\end{minipage}
\begin{minipage}{0.48\textwidth}
\centering
\begin{tikzpicture}[scale=0.92]
\draw[ultra thick,-latex](0,0)--(-4,-1);\node[above]at(-4,-0.8){$\substack{7^2,7^3,\dotsc,7^{q-1}}$};
\draw(0,0)--(-2.75,-1)--(-2.25,-1)--(0,0);\node[below]at(-2.5,-1){$\substack{\{3,5\}\times7\\4\text{ branches}}$};
\draw(0,0)--(-1,-1);\node[below]at(-1,-1){$\substack{p\times3\times7}$};
\draw(0,0)--(1,-1)--(-1.125,-2)--(-0.875,-2)--(1,-1)--(0.625,-2)--(0.875,-2)--(1,-1)--(2.375,-2)--(2.625,-2)--(1,-1);\draw[ultra thick,-latex](1,-1)--(-2.5,-2);\node[below]at(-3.3,-2){$\substack{11^2,11^3,\dotsc,11^{q-1}}$};
\node[below]at(-1,-2){$\substack{\{3,7\}\times5\times11\\4\text{ branches}}$};\node[below]at(0.75,-2){$\substack{\{3,7\}\times11\\4\text{ branches}}$};\node[below]at(2.5,-2){$\substack{\{3\}\times p\times7\times11\\2\text{ branches}}$};
\end{tikzpicture}
\caption{$T_4$ in Figure~\ref{fig:p-5ps}}
\label{fig:p-5psT4}
\end{minipage}
\end{figure}

\begin{figure}[H]
\centering
\begin{tikzpicture}[scale=0.98]
\draw[ultra thick,-latex](0,0)--(-8,-1);\node[above]at(-8,-0.9){$\substack{7^2,7^3,\dotsc,7^{q-1}}$};
\draw(0,0)--(-6.8,-1)--(-6.3,-1)--(0,0);\node[below]at(-6.8,-1){$\substack{\{3\}\times7\\2\text{ branches}}$};

\draw(0,0)--(-5,-1);\draw[ultra thick,-latex](-5,-1)--(-6.5,-2);\node at(-7.8,-2){$\substack{13^2,13^3,\dotsc,13^{q-1}}$};
\draw(-5,-1)--(-5.625,-2)--(-5.375,-2)--(-5,-1);\node[below]at(-5.7,-2){$\substack{\{p,3,5\}\times13\\8\text{ branches}}$};
\draw(-5,-1)--(-4.375,-2)--(-4.125,-2)--(-5,-1);\node[below]at(-4,-2){$\substack{\{3,5\}\times7\times13\\4\text{ branches}}$};

\draw(0,0)--(-1,-1);\draw[ultra thick,-latex](-1,-1)--(-2.8,-2);\node[above]at(-3.1,-1.7){$\substack{13^2,13^3,\dotsc,13^{q-1}}$};
\draw(-1,-1)--(-1.875,-2)--(-1.625,-2)--(-1,-1);\node[below]at(-1.9,-2){$\substack{\{p,3,5\}\times13\\8\text{ branches}}$};
\draw(-1,-1)--(-0.125,-2)--(0.125,-2)--(-1,-1);\node[below]at(-0,-2){$\substack{\{3,5\}\times p\times7\times13\\4\text{ branches}}$};

\draw(0,0)--(1,-1);\node[below]at(1,-1){$\substack{p\times7}$};
\draw(0,0)--(3,-1);\node[below]at(3,-1){$\substack{T_7}$};
\end{tikzpicture}
\caption{$T_5$ in Figure~\ref{fig:p-5ps}}
\label{fig:p-5psT5}
\end{figure}

\begin{figure}[H]
\centering
\begin{tikzpicture}[scale=0.98]
\draw[ultra thick,-latex](0,0)--(-8,-1);\node[above]at(-8,-0.9){$\substack{11^2,11^3,\dotsc,11^{q-1}}$};
\draw(0,0)--(5.25,-1)--(5.75,-1)--(0,0);\node[below]at(5.6,-1){$\substack{\{3,7\}\times11\\4\text{ branches}}$};
\draw(0,0)--(6.25,-1)--(6.75,-1)--(0,0);\node[above]at(6.2,-0.9){$\substack{\{3\}\times p\times11\\2\text{ branches}}$};
\draw(0,0)--(-6,-1)--(-7.625,-2)--(-7.375,-2)--(-6,-1)--(-6.125,-2)--(-5.875,-2)--(-6,-1);\draw[ultra thick,-latex](-6,-1)--(-8.3,-2);
\node[above]at(-8.1,-1.6){$\substack{13^2,13^3,\dotsc,13^{q-1}}$};\node[below]at(-7.5,-2){$\substack{\{p,3,5\}\times13\\8\text{ branches}}$};\node[below]at(-6,-2){$\substack{\{p,3\}\\\times11\times13\\4\text{ branches}}$};

\draw(0,0)--(-2.5,-1)--(-4.125,-2)--(-3.875,-2)--(-2.5,-1)--(-2.625,-2)--(-2.375,-2)--(-2.5,-1);\draw[ultra thick,-latex](-2.5,-1)--(-4.8,-2);
\node[above]at(-4.6,-1.6){$\substack{13^2,13^3,\dotsc,13^{q-1}}$};\node[below]at(-4.2,-2){$\substack{\{p,3,5\}\times13\\8\text{ branches}}$};\node[below]at(-2.5,-2){$\substack{\{p,3\}\\\times5\times11\times13\\4\text{ branches}}$};

\draw(0,0)--(1,-1)--(-0.625,-2)--(-0.375,-2)--(1,-1)--(0.875,-2)--(1.125,-2)--(1,-1);\draw[ultra thick,-latex](1,-1)--(-1.3,-2);
\node[above]at(-1.1,-1.6){$\substack{13^2,13^3,\dotsc,13^{q-1}}$};\node[below]at(-0.7,-2){$\substack{\{p,3,5\}\times13\\8\text{ branches}}$};\node[below]at(1,-2){$\substack{\{p,3\}\\\times7\times11\times13\\4\text{ branches}}$};

\draw(0,0)--(4.5,-1)--(2.875,-2)--(3.125,-2)--(4.5,-1)--(4.575,-2)--(4.825,-2)--(4.5,-1);\draw[ultra thick,-latex](4.5,-1)--(2.2,-2);
\node[above]at(2.4,-1.6){$\substack{13^2,13^3,\dotsc,13^{q-1}}$};\node[below]at(2.8,-2){$\substack{\{p,3,5\}\times13\\8\text{ branches}}$};\node[below]at(4.7,-2){$\substack{\{p,3\}\\\times5\times7\times11\times13\\4\text{ branches}}$};
\end{tikzpicture}
\caption{$T_7$ in Figure~\ref{fig:p-5psT5}}
\label{fig:p-5psT7}
\end{figure}

\begin{figure}[H]
\centering
\begin{tikzpicture}[scale=0.98]
\draw[ultra thick,-latex](0,0)--(-8,-1);\node[above]at(-8,-0.9){$\substack{7^2,7^3,\dotsc,7^{q-1}}$};
\draw(0,0)--(-6.8,-1)--(-6.3,-1)--(0,0);\node[below]at(-6.8,-1){$\substack{\{3\}\times7\\2\text{ branches}}$};

\draw(0,0)--(-5,-1);\draw[ultra thick,-latex](-5,-1)--(-6.5,-2);\node at(-7.8,-2){$\substack{17^2,17^3,\dotsc,17^{q-1}}$};
\draw(-5,-1)--(-5.725,-2)--(-5.475,-2)--(-5,-1);\node[below]at(-5.8,-2){$\substack{\{p,3,5\}\times17\\8\text{ branches}}$};
\draw(-5,-1)--(-4.375,-2)--(-4.125,-2)--(-5,-1);\node[below]at(-4,-2){$\substack{\{p,3,5\}\times7\times17\\8\text{ branches}}$};

\draw(0,0)--(-1,-1);\draw[ultra thick,-latex](-1,-1)--(-2.8,-2);\node[above]at(-3.1,-1.7){$\substack{19^2,19^3,\dotsc,19^{q-1}}$};
\draw(-1,-1)--(-1.875,-2)--(-1.625,-2)--(-1,-1);\node[below]at(-1.9,-2){$\substack{\{p,3,5\}\times19\\8\text{ branches}}$};
\draw(-1,-1)--(-0.125,-2)--(0.125,-2)--(-1,-1);\node[below]at(-0,-2){$\substack{\{p,3,5\}\times7\times19\\8\text{ branches}}$};

%\draw(-1,-1)--(4.5,-2);\draw[ultra thick,-latex](4.5,-2)--(-0.5,-3);\node[below]at(-1.5,-3){$\substack{17^2,17^3,\dotsc,17^{q-1}}$};
%\draw(4.5,-2)--(0.625,-3)--(0.875,-3)--(4.5,-2);\node[below]at(0.4,-3){$\substack{\{p,3,5\}\\\times17\\8\text{ branches}}$};
%\draw(4.5,-2)--(1.875,-3)--(2.125,-3)--(4.5,-2);\node[below]at(2,-3){$\substack{\{p,3,5\}\\\times19\times17\\8\text{ branches}}$};
\draw(-1,-1)--(1.5,-2);\node[below]at(1.5,-2){$\substack{T_9}$};

%\draw(-1,-1)--(6.5,-2);\draw[ultra thick,-latex](6.5,-2)--(3.5,-3);\node[above,rotate=(atan(1/3))]at(4.25,-2.75){$\substack{17^2,17^3,\dotsc,17^{q-1}}$};
%\draw(6.5,-2)--(4.375,-3)--(4.625,-3)--(6.5,-2);\node[below]at(4.5,-3){$\substack{\{p,3,5\}\\\times17\\8\text{ branches}}$};
%\draw(6.5,-2)--(6.125,-3)--(6.375,-3)--(6.5,-2);\node[below]at(6.25,-3){$\substack{\{p,3,5\}\\\times7\times19\times17\\8\text{ branches}}$};
\draw(-1,-1)--(3.5,-2);\draw[ultra thick,-latex](3.5,-2)--(1.5,-3);\node[below]at(1,-3){$\substack{17^2,17^3,\dotsc,17^{q-1}}$};
\draw(3.5,-2)--(3.175,-3)--(3.425,-3)--(3.5,-2);\node[below]at(3.3,-3){$\substack{\{p,3,5\}\times17\\8\text{ branches}}$};
\draw(3.5,-2)--(5.375,-3)--(5.625,-3)--(3.5,-2);\node[below]at(5.5,-3){$\substack{\{p,3,5\}\times7\times19\times17\\8\text{ branches}}$};

\draw(0,0)--(4,-1);\node[below]at(4,-1){$\substack{p\times3\times7}$};
\draw(0,0)--(6,-1);\node[below]at(6,-1){$\substack{T_8}$};
\end{tikzpicture}
\caption{$T_6$ in Figure~\ref{fig:p-5ps}}
\label{fig:p-5psT6}
\end{figure}

\begin{figure}[H]
\centering
\begin{tikzpicture}[scale=0.95]
\draw[ultra thick,-latex](0,0)--(-8.3,-1);\node[above]at(-8,-0.9){$\substack{11^2,11^3,\dotsc,11^{q-1}}$};
\draw(0,0)--(5.25,-1)--(5.75,-1)--(0,0);\node[below]at(6,-1){$\substack{\{3,7\}\times11\\4\text{ branches}}$};
\draw(0,0)--(6.25,-1)--(6.75,-1)--(0,0);\node[above]at(6.2,-0.9){$\substack{\{3\}\times p\times7\times11\\2\text{ branches}}$};
\draw(0,0)--(-6.5,-1)--(-8.125,-2)--(-7.875,-2)--(-6.5,-1)--(-6.625,-2)--(-6.375,-2)--(-6.5,-1);\draw[ultra thick,-latex](-6.5,-1)--(-8.8,-2);
\node[above,rotate=(atan(1/2.3))]at(-8.455,-1.85){$\substack{17^2,17^3,\dotsc,17^{q-1}}$};\node[below]at(-8,-2){$\substack{\{p,3,5\}\times17\\8\text{ branches}}$};\node[below]at(-6.5,-2){$\substack{\{p,3,5\}\\\times11\times17\\8\text{ branches}}$};

\draw(0,0)--(-3.3,-1)--(-4.925,-2)--(-4.675,-2)--(-3.3,-1)--(-3.425,-2)--(-3.175,-2)--(-3.3,-1);\draw[ultra thick,-latex](-3.3,-1)--(-5.6,-2);
\node[above,rotate=(atan(1/2.3))]at(-4.83,-1.67){$\substack{17^2,17^3,\dotsc,17^{q-1}}$};\node[below]at(-4.95,-2){$\substack{\{p,3,5\}\times17\\8\text{ branches}}$};\node[below]at(-3.3,-2){$\substack{\{p,3,5\}\\\times7\times11\times17\\8\text{ branches}}$};

\draw(0,0)--(-0.1,-1)--(-1.725,-2)--(-1.475,-2)--(-0.1,-1)--(-0.225,-2)--(0.025,-2)--(-0.1,-1);\draw[ultra thick,-latex](-0.1,-1)--(-2.4,-2);
\node[above,rotate=(atan(1/2.3))]at(-1.57,-1.67){$\substack{19^2,19^3,\dotsc,19^{q-1}}$};\node[below]at(-1.6,-2){$\substack{\{p,3,5\}\times19\\8\text{ branches}}$};\node[below]at(-0.1,-2){$\substack{\{p,3,5\}\\\times11\times19\\8\text{ branches}}$};
\draw(-0.1,-1)--(0.6,-2);\node[below]at(0.7,-2){$\substack{T_9}$};
\draw(-0.1,-1)--(1.2,-2);\node[below]at(1.3,-2){$\substack{T_{10}}$};

\draw(0,0)--(4.3,-1)--(2.675,-2)--(2.925,-2)--(4.3,-1)--(4.175,-2)--(4.425,-2)--(4.3,-1);\draw[ultra thick,-latex](4.3,-1)--(2,-2);
\node[above,rotate=(atan(1/2.3))]at(2.675,-1.75){$\substack{19^2,19^3,\dotsc,19^{q-1}}$};\node[below]at(2.7,-2){$\substack{\{p,3,5\}\times19\\8\text{ branches}}$};\node[below]at(4.3,-2){$\substack{\{p,3,5\}\\\times7\times11\times19\\8\text{ branches}}$};
\draw(4.3,-1)--(5.1,-2);\node[below]at(5.2,-2){$\substack{T_9}$};
\draw(4.3,-1)--(5.6,-2);\node[below]at(5.7,-2){$\substack{T_{11}}$};
\end{tikzpicture}
\caption{$T_8$ in Figure~\ref{fig:p-5psT6}}
\label{fig:p-5psT8}
\end{figure}

\begin{figure}[H]
\centering
\begin{minipage}{0.32\textwidth}
\centering
\begin{tikzpicture}[scale=1]
\draw[ultra thick,-latex](0,0)--(-2,-1);\node[above]at(-2.4,-0.8){$\substack{17^2,17^3,\dotsc,17^{q-1}}$};
\draw(0,0)--(-1.125,-1)--(-0.875,-1)--(0,0);\node[below]at(-1,-1){$\substack{\{p,3,5\}\times17\\8\text{ branches}}$};
\draw(0,0)--(0.575,-1)--(0.825,-1)--(0,0);\node[below]at(0.7,-1){$\substack{\{p,3,5\}\\\times19\times17\\8\text{ branches}}$};
\end{tikzpicture}
\caption{$T_9$ in Figure~\ref{fig:p-5psT6} and Figure~\ref{fig:p-5psT8}}
\label{fig:p-5psT9}
\end{minipage}
\begin{minipage}{0.32\textwidth}
\centering
\begin{tikzpicture}[scale=0.92]
\draw[ultra thick,-latex](0,0)--(-2,-1);\node[above]at(-2.4,-0.8){$\substack{17^2,17^3,\dotsc,17^{q-1}}$};
\draw(0,0)--(-1.125,-1)--(-0.875,-1)--(0,0);\node[below]at(-1,-1){$\substack{\{p,3,5\}\times17\\8\text{ branches}}$};
\draw(0,0)--(0.875,-1)--(1.125,-1)--(0,0);\node[below]at(1,-1){$\substack{\{p,3,5\}\\\times11\times19\times17\\8\text{ branches}}$};
\end{tikzpicture}
\caption{$T_{10}$ in Figure~\ref{fig:p-5psT8}}
\label{fig:p-5psT10}
\end{minipage}
\begin{minipage}{0.32\textwidth}
\centering
\begin{tikzpicture}[scale=0.92]
\draw[ultra thick,-latex](0,0)--(-2,-1);\node[above]at(-2.4,-0.8){$\substack{17^2,17^3,\dotsc,17^{q-1}}$};
\draw(0,0)--(-1.125,-1)--(-0.875,-1)--(0,0);\node[below]at(-1,-1){$\substack{\{p,3,5\}\times17\\8\text{ branches}}$};
\draw(0,0)--(0.875,-1)--(1.125,-1)--(0,0);\node[below]at(1,-1){$\substack{\{p,3,5\}\\\times7\times11\times19\times17\\8\text{ branches}}$};
\end{tikzpicture}
\caption{$T_{11}$ in Figure~\ref{fig:p-5psT8}}
\label{fig:p-5psT11}
\end{minipage}
\end{figure}
\end{proof}

\section{Extension of results and concluding remarks}\label{sec:twolemmas}

We can extend Theorems~\ref{thm:tau7} and \ref{thm:t11} by using the following lemma.

\begin{lemma}\label{lem:squarefree}
Let $p$ be a prime and $t$ be a positive integer such that $t\leq p$. Let $\mathcal{C}_0$ be a covering system of the integers whose tree diagram is given in Figure~$\ref{fig:rootp-node}$, where all moduli are distinct except that $p$ is used exactly $p-t$
times as a modulus, and no modulus of $\mathcal{C}_0$ is divisible by $p^2$.
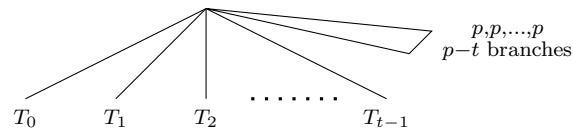
\begin{figure}[H]
\centering
\begin{tikzpicture}[scale=1.2]
\draw(0,0)--(2.5,-0.25)--(2.25,-0.5)--(0,0);
\foreach[count=\j]\i in{-2,-1,0,2}{\draw(0,0)--(\i,-1);}\draw[very thick,loosely dotted](0.5,-1)--(1.5,-1);
\foreach[count=\j]\i in{-2,-1,0}{\tikzmath{\x=int(\j-1));}\node[below]at(\i,-1){$\substack{T_\x}$};}\node[below]at(2,-1){$\substack{T_{t-1}}$};
\node[right]at(2.5,-0.375){$\substack{p,p,\dotsc,p\\p-t\text{ branches}}$};
%\node[below]at(0,-1.2){$\underbrace{\hspace{140pt}}_{t\text{ branches}}$};
\end{tikzpicture}
\caption{The root of a tree diagram for $\mathcal{C}_0$}
\label{fig:rootp-node}
\end{figure}
\noindent Then for all primes $q>p$, there exists a covering system $\mathcal{C}$ of the integers such that all moduli are distinct except that $q$ is used exactly $q-t$ times as a modulus, and no modulus of $\mathcal{C}$ is divisible by $q^2$. Furthermore, if all moduli of $\mathcal{C}_0$ are odd, then all moduli of $\mathcal{C}$ are odd; if all moduli of $\mathcal{C}_0$ are square-free, then all moduli of $\mathcal{C}$ are square-free.
\end{lemma}

\begin{proof}
In the tree diagram of $\mathcal{C}_0$ as shown in Figure~\ref{fig:rootp-node}, the root is a $p$-node, and for each $i\in\{0,1,\dotsc,t-1\}$, $T_i$ denotes the subtree below the node that corresponds to the congruence $\mod{i}{p}$. Note that this is the only $p$-node in the tree since no modulus of $\mathcal{C}_0$ is divisible by $p^2$.

To obtain a tree diagram for $\mathcal{C}$, we are going to replace the root of $\mathcal{C}_0$ by a $q$-node as shown in Figure~\ref{fig:rootq-node}.
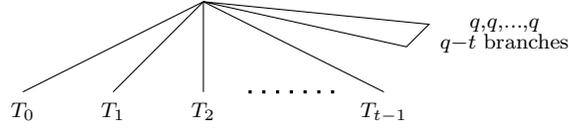
\begin{figure}[H]
\centering
\begin{tikzpicture}[scale=1.2]
\draw(0,0)--(2.5,-0.25)--(2.25,-0.5)--(0,0);
\foreach[count=\j]\i in{-2,-1,0,2}{\draw(0,0)--(\i,-1);}\draw[very thick,loosely dotted](0.5,-1)--(1.5,-1);
\foreach[count=\j]\i in{-2,-1,0}{\tikzmath{\x=int(\j-1));}\node[below]at(\i,-1){$\substack{T_\x}$};}\node[below]at(2,-1){$\substack{T_{t-1}}$};
\node[right]at(2.5,-0.375){$\substack{q,q,\dotsc,q\\q-t\text{ branches}}$};
\end{tikzpicture}
\caption{The root of a tree diagram for $\mathcal{C}$}
\label{fig:rootq-node}
\end{figure}
\noindent The only difficulty in this root replacement is that it may not be compatible with the existing $q$-nodes in the tree diagram for $\mathcal{C}_0$. To solve this issue, we replace every $q$-node in the tree diagram for $\mathcal{C}_0$ (Figure~\ref{fig:otherq-node}) by a $p$-node (Figure~\ref{fig:otherp-node}) and keep only the subtrees $T'_i$ for $i\in\{0,1,\dotsc,p-1\}$. Lastly, if a modulus $m$ in $\mathcal{C}_0$ satisfies $m=pr$ such that $p\nmid r$, then the resultant modulus in $\mathcal{C}$ is $qr$; if a modulus $m$ in $\mathcal{C}_0$ satisfies $m=q^\alpha s$ for some positive integer $\alpha$ such that $q\nmid s$, then the resultant modulus in $\mathcal{C}$ is $p^\alpha s$. This completes our construction of a tree diagram for $\mathcal{C}$.
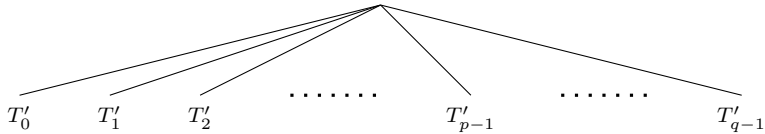
\begin{figure}[H]
\centering
\begin{tikzpicture}[scale=1.2]
\foreach\i in{-4,-3,-2,1,4}{\draw(0,0)--(\i,-1);}\draw[very thick,loosely dotted](-1,-1)--(0,-1);\draw[very thick,loosely dotted](2,-1)--(3,-1);
\foreach[count=\j]\i in{-4,-3,-2}{\tikzmath{\x=int(\j-1));}\node[below]at(\i,-1){$\substack{T'_\x}$};}\node[below]at(1,-1){$\substack{T'_{p-1}}$};\node[below]at(4,-1){$\substack{T'_{q-1}}$};
\end{tikzpicture}
\caption{A $q$-node in the tree diagram for $\mathcal{C}_0$}
\label{fig:otherq-node}
\end{figure}
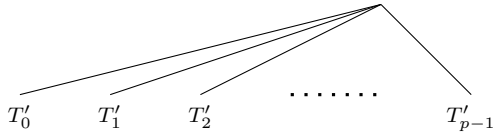
\begin{figure}[H]
\centering
\begin{tikzpicture}[scale=1.2]
\foreach\i in{-4,-3,-2,1}{\draw(0,0)--(\i,-1);}\draw[very thick,loosely dotted](-1,-1)--(0,-1);
\foreach[count=\j]\i in{-4,-3,-2}{\tikzmath{\x=int(\j-1));}\node[below]at(\i,-1){$\substack{T'_\x}$};}\node[below]at(1,-1){$\substack{T'_{p-1}}$};
\end{tikzpicture}
\caption{A $p$-node in the tree diagram for $\mathcal{C}$}
\label{fig:otherp-node}
\end{figure}
\end{proof}

%Combining Lemma~\ref{lem:squarefree} with Theorems~\ref{thm:tau7} and \ref{thm:t11}, we have the following results.
Combining Lemma~\ref{lem:squarefree} with Theorem~\ref{thm:tau7}, we have the following result.

\begin{corollary}
For all primes $p\geq7$, $\tau_p\leq p-1$.
\end{corollary}

As for the bounds on $t_p$ for primes $p\geq3$, we summarize the results in the following table.

\begin{table}[H]
\centering
\begin{tabular}{|c|c|c|}
\hline
$p$& Upper bound on $t_p$& Sources\\
\hline
$3$& $t_p\leq2=p-1$& \cite{h}\\
$5$& $t_p\leq3=p-2$& \cite{hhm}\\
$7$& $t_p\leq4=p-3$& Theorem~\ref{thm:t7}\\
$11\leq p\leq19$& $t_p\leq p-4$& Theorem~\ref{thm:t11} \& Lemma~\ref{lem:squarefree}\\
$p\geq23$& $t_p\leq p-5$& Theorem~\ref{thm:t23}\\
\hline
\end{tabular}
\caption{Summary of results regarding Question~\ref{question:mainquestion}}
\label{table}
\end{table}

We observe from Table~\ref{table} that there exists a constant $c$ such that for all sufficiently large primes $p$, $t_p\leq p-c$. For future investigation, it is certainly of interest to continue improving on the constant $c$. Nevertheless, it is of greater interest to investigate the following question.

\begin{question}\label{question:final}
Does there exist a constant $0\leq\epsilon<1$ such that for all sufficiently large primes $p$, $t_p\leq\epsilon p$?
\end{question}

The existence of an odd covering system would imply that we can take $\epsilon=0$ for all sufficiently large primes $p$, thus Question~\ref{question:final} provides a progressive approach to solving the odd covering problem.

Before we end this paper, we also provide a similar lemma to Lemma~\ref{lem:squarefree}, which could be useful in the future.

\begin{lemma}\label{lem:nonsquarefree}
Let $p$ be a prime and $t$ be a positive integer such that $t\leq p$. Let $\mathcal{C}_0$ be a covering system of the integers whose tree diagram is given in Figure~$\ref{fig:rootp-node2}$, where all moduli are distinct except that $p$ is used exactly $p-t$ times as a modulus.
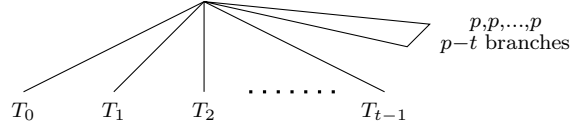
\begin{figure}[H]
\centering
\begin{tikzpicture}[scale=1.2]
\draw(0,0)--(2.5,-0.25)--(2.25,-0.5)--(0,0);
\foreach[count=\j]\i in{-2,-1,0,2}{\draw(0,0)--(\i,-1);}\draw[very thick,loosely dotted](0.5,-1)--(1.5,-1);
\foreach[count=\j]\i in{-2,-1,0}{\tikzmath{\x=int(\j-1));}\node[below]at(\i,-1){$\substack{T_\x}$};}\node[below]at(2,-1){$\substack{T_{t-1}}$};
\node[right]at(2.5,-0.375){$\substack{p,p,\dotsc,p\\p-t\text{ branches}}$};
%\node[below]at(0,-1.2){$\underbrace{\hspace{140pt}}_{t\text{ branches}}$};
\end{tikzpicture}
\caption{The root of a tree diagram for $\mathcal{C}_0$}
\label{fig:rootp-node2}
\end{figure}
\noindent Let $M$ be the least common multiple of all moduli of $\mathcal{C}_0$. Then for all odd primes $q\geq t$ that satisfies $q\nmid M$, there exists a covering system $\mathcal{C}$ of the integers such that all moduli are distinct except that $q$ is used exactly $q-t$ times as a modulus. Furthermore, if all moduli of $\mathcal{C}_0$ are odd, then all moduli of $\mathcal{C}$ are odd; if all moduli of $\mathcal{C}_0$ are square-free, then all moduli of $\mathcal{C}$ are square-free.
\end{lemma}

\begin{proof}
In the tree diagram of $\mathcal{C}_0$ as shown in Figure~\ref{fig:rootp-node2}, the root is a $p$-node, and for each $i\in\{0,1,\dotsc,t-1\}$, $T_i$ denotes the subtree below the node that corresponds to the congruence $\mod{i}{p}$.

To obtain a tree diagram for $\mathcal{C}$, we are going to replace the root of $\mathcal{C}_0$ by a $q$-node as shown in Figure~\ref{fig:rootq-node2}.
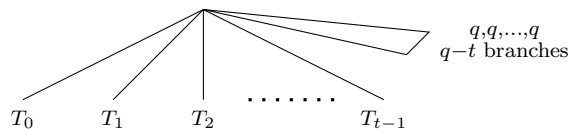
\begin{figure}[H]
\centering
\begin{tikzpicture}[scale=1.2]
\draw(0,0)--(2.5,-0.25)--(2.25,-0.5)--(0,0);
\foreach[count=\j]\i in{-2,-1,0,2}{\draw(0,0)--(\i,-1);}\draw[very thick,loosely dotted](0.5,-1)--(1.5,-1);
\foreach[count=\j]\i in{-2,-1,0}{\tikzmath{\x=int(\j-1));}\node[below]at(\i,-1){$\substack{T_\x}$};}\node[below]at(2,-1){$\substack{T_{t-1}}$};
\node[right]at(2.5,-0.375){$\substack{q,q,\dotsc,q\\q-t\text{ branches}}$};
\end{tikzpicture}
\caption{The root of a tree diagram for $\mathcal{C}$}
\label{fig:rootq-node2}
\end{figure}
\noindent There are no existing $q$-nodes in the tree diagram for $\mathcal{C}_0$, so we do not need to perform any further replacement, and have already obtained a tree diagram for $\mathcal{C}$. Nonetheless, it is worth noting that if a modulus $m$ in $\mathcal{C}_0$ satisfies $m=p^\alpha r$ for some positive integer $\alpha$ such that $p\nmid r$, then the resultant modulus in $\mathcal{C}$ is $qp^{\alpha-1}r$.
\end{proof}

%\section{Concluding Remarks}
%In Section~\ref{sec:squarefree}, we showed that if $p\geq 7$, then $\tau_p\leq p-1$.  On the other hand, we also showed that if $\tau_3\leq 2$, then there exists an odd covering system of the integers.  Unfortunately, our investigations do not establish bounds for $\tau_5$.  Obviously, if a square-free odd covering the integers exists, then $\tau_5\leq 1$.  Nonetheless, establishing $\tau_5\leq 4$ seems a difficult task.

\end{document}